\documentclass
[
    a4paper,
    DIV=14,
    abstract=true,
    numbers=noenddot
]
{scrartcl}

\usepackage
{
    graphicx,
    amssymb,
    amsmath,
    amsthm,
    dsfont, 
    mathtools,
    authblk,
    enumitem,
    nicefrac,
}

\usepackage{quantikz}

\usepackage[utf8]{inputenc}

\usepackage[pdffitwindow=false,
            plainpages=false,
            pdfpagelabels=true,
            pdfpagemode=UseOutlines,
            pdfpagelayout=SinglePage,
            bookmarks=false,
            colorlinks=true,
            hyperfootnotes=false,
            linkcolor=blue,
            urlcolor=blue!30!black,
            citecolor=green!50!black]{hyperref}

\usepackage[bf,normal]{caption}

\graphicspath{{pics/}}

\newcommand{\subfiguretitle}[1]{{\scriptsize{#1}} \\}
\newcommand{\R}{\mathbb{R}}                                     
\newcommand{\C}{\mathbb{C}}                                     
\newcommand{\pd}[2]{\frac{\partial#1}{\partial#2}}              
\newcommand{\innerprod}[2]{\left\langle #1,\, #2 \right\rangle} 
\newcommand{\ts}{\hspace*{0.1em}}                               
\newcommand{\diff}{\,\mathrm{d}}                                
\providecommand{\abs}[1]{\left\lvert #1 \right\rvert}           
\providecommand{\norm}[1]{\left\lVert #1 \right\rVert}          
\providecommand{\vdot}{\boldsymbol\cdot}                        

\newcommand\restr[2]{{\left.\kern-\nulldelimiterspace #1\vphantom{\big|}\right|_{#2}}}

\newcommand\xqed[1]{\leavevmode\unskip\penalty9999 \hbox{}\nobreak\hfill \quad\hbox{#1}}
\newcommand{\exampleSymbol}{\xqed{$\triangle$}}

\DeclareMathOperator{\tr}{tr}
\DeclareMathOperator{\mspan}{span}
\let\div\relax
\DeclareMathOperator{\div}{div}

\newtheorem{theorem}{Theorem}[section]
\newtheorem{corollary}[theorem]{Corollary}
\newtheorem{lemma}[theorem]{Lemma}

\newtheorem{definition}[theorem]{Definition}
\theoremstyle{definition}
\newtheorem{example}[theorem]{Example}
\newtheorem{remark}[theorem]{Remark}

\makeatletter
\renewcommand*\env@matrix[1][*\c@MaxMatrixCols c]{%
  \hskip -\arraycolsep
  \let\@ifnextchar\new@ifnextchar
  \array{#1}}
\makeatother

\mathtoolsset{centercolon} 

\allowdisplaybreaks

\DeclareCaptionLabelFormat{period}{#1~#2}
\makeatletter
\def\blfootnote{\gdef\@thefnmark{}\@footnotetext}
\makeatother

\begin{document}

\title{Numerical approximation of the \\ Koopman--von Neumann equation: \\ Operator learning and quantum computing}
\author[1]{Stefan Klus}
\author[2]{Feliks Nüske}
\author[3]{Patrick Gelß}
\affil[1]{School of Mathematical \& Computer Sciences, Heriot--Watt University, Edinburgh, UK}
\affil[2]{Max-Planck-Institute for Dynamics of Complex Technical Systems, Magdeburg, Germany}
\affil[3]{Zuse Institute Berlin, Berlin, Germany}

\date{}

\maketitle

\begin{abstract}
The Koopman--von Neumann equation describes the evolution of wavefunctions associated with autonomous ordinary differential equations and can be regarded as a quantum physics-inspired formulation of classical mechanics. The main advantage compared to conventional transfer operators such as Koopman and Perron--Frobenius operators is that the Koopman--von Neumann operator is unitary even if the dynamics are non-Hamiltonian. Projecting this operator onto a finite-dimensional subspace allows us to represent it by a unitary matrix, which in turn can be expressed as a quantum circuit. We will exploit relationships between the Koopman--von Neumann framework and classical transfer operators in order to derive numerical methods to approximate the Koopman--von Neumann operator and its eigenvalues and eigenfunctions from data. Furthermore, we will show that the choice of basis functions and domain are crucial to ensure that the operator is well-defined. We will illustrate the results with the aid of guiding examples, including simple undamped and damped oscillators and the Lotka--Volterra model.
\end{abstract}

\section{Introduction}

In recent years, the field of quantum computing has advanced rapidly. Exploiting quantum mechanical effects such as superposition and entanglement, quantum computers are able to solve certain types of problems faster than their classical counterparts. More specifically, quantum speedups have been rigorously established for a limited number of structured problem classes, most notably integer factorization, discrete logarithms, and related order-finding tasks~\cite{Shor1997, Montanaro2016}, whereas unstructured search and several Monte Carlo-type estimation problems admit at most quadratic quantum improvements~\cite{Grover1996, Brassard2002}. There is currently no evidence that quantum computers can efficiently solve general NP-hard problems. Moreover, many other proposed quantum advantages rely on additional assumptions on data access, sparsity, conditioning, or output representation~\cite{Harrow2009, Dalzell2025}. A natural question thus is whether quantum computers can be used to speed up the simulation and analysis of classical nonlinear dynamical systems. One possible approach is to utilize the Koopman--von Neumann framework, a quantum physics-inspired formulation of classical mechanics \cite{Mauro02, BB14, Klein2018, Joseph20, Mor20, LLASS22}. While Koopman and von Neumann originally considered only Hamiltonian systems---in this case the Koopman operator itself is already unitary---, the approach has been extended to arbitrary autonomous ordinary differential equations. Instead of analyzing the Koopman operator, which propagates observables, or its adjoint, the Perron--Frobenius operator, which propagates probability densities, it is possible to derive a unitary operator, the Koopman--von Neumann operator,\!\footnote{We will call the propagator for a fixed lag time \emph{Koopman--von Neumann operator} and the associated infinitesimal generator of the semi-group of operators \emph{Koopman--von Neumann generator} so that this is consistent with the notation used for the other operators and their generators.} that describes the evolution of complex-valued wavefunctions. The corresponding probability densities can then be obtained by applying Born's rule.

Projecting the Koopman--von Neumann operator onto a space spanned by a set of basis functions allows us to represent it by a unitary matrix, which in turn can be encoded by a quantum circuit. It is thus possible to simulate the action of the Koopman--von Neumann operator on wavefunctions using quantum computing. However, there are important differences between the behavior of the Koopman--von Neumann equation and the standard Schrödinger equation as pointed out in~\cite{Joseph20}. Most importantly, the Koopman--von Neumann generator is just a first-order differential operator unlike the Schrödinger operator. Furthermore, neither Planck's constant $ \hbar $ nor a phase is involved in this fully
classical theory \cite{LLASS22}. We will analyze the feasibility of Koopman--von Neumann mechanics for describing the evolution of deterministic dynamical systems---including non-Hamiltonian dynamics---and explore relationships with the closely related Koopman and Perron--Frobenius operators and their infinitesimal generators.

While conventional transfer operators such as the Koopman operator and Perron--Frobenius operator are in general well understood and have been used extensively (and successfully) in the past, see \cite{LaMa94, DJ99, Mezic05, BMM12, Mezic13, WKR15, KKS16, Gia19, Mez23} to name just a few, the Koopman--von Neumann framework remains largely unexplored and almost no real-world applications can be found in the literature. Rigorous functional-analytic studies of the Koopman--von Neumann equation have emerged only recently. In particular, the well-posedness of the associated initial value problem for dynamical systems on bounded domains has been investigated in \cite{SGKP24}, clarifying the relationships with classical transport equations and transfer operator theory.
Further works study the Koopman--von Neumann framework from an operator-theoretic perspective and consider structural properties of the corresponding Liouville operators in specific dynamical settings \cite{McCaul2019, Grotto2020}. Another active line of research investigates the Koopman--von Neumann formalism in an algorithmic setting, addressing implementations on quantum hardware and in the form of circuit-based simulations \cite{Joseph20, Simon2024, Cochran2025, Novikau2025}. These results establish important theoretical foundations and provide a starting point for the development and analysis of numerical approximation schemes. Classical numerical simulation methods and data-driven techniques for the analysis of the Koopman--von Neumann equation have received little attention in the literature so far.

We are in particular interested in spectral properties since the eigenvalues and eigenfunctions of conventional transfer operators contain important information about dominant timescales or slowly evolving spatiotemporal modes, but also in the propagation of probability densities, observables, and wavefunctions. Although the Koopman operator has been known for a long time, its recent rise in popularity can be explained by the availability of efficient machine learning algorithms for estimating transfer operators from data and the abundance of training data. This allows us to gain insights into global properties of the underlying system without requiring detailed mathematical models. One of the most frequently used methods for approximating transfer operators is \emph{extended dynamic mode decomposition} \cite{WKR15, KKS16}, which can be regarded as a nonlinear variant of \emph{dynamic mode decomposition} \cite{Schmid10, TRLBK14}. There are various kernel-based extensions \cite{WRK15, KSM19}, deep learning formulations~\cite{LDBK17, MPWN18}, and methods to directly approximate infinitesimal generators of transfer operators~\cite{MauGon16, KNPNCS20}. Closely related methods have been developed independently by the signal processing, climate science, and molecular dynamics communities, see \cite{MS94, PS95, NoNu13}. Comprehensive reviews of DMD-type algorithms can be found in \cite{KNKWKSN18, Colbrook24}. Our goal is to show how variants of these methods can also be used to approximate the Koopman--von Neumann operator or its generator. The main contributions of this work are as follows:
\begin{enumerate}[leftmargin=4ex, topsep=0.5ex, itemsep=0ex]
\item We analyze relationships between infinitesimal generators of conventional transfer operators and the Koopman--von Neumann framework and derive Galerkin projections of these operators.
\item We approximate the Koopman--von Neumann generator using data-driven methods as well as finite element methods, which enables us not only to compute eigenvalues and eigenfunctions, but also to propagate wavefunctions and thus probability densities.
\item In order to illustrate the results and to highlight the advantages and limitations of the proposed methods, we consider simple guiding examples and benchmark problems, including undamped and damped oscillators as well as the Lotka--Volterra model.
\item We show that choosing suitable basis functions for the undamped oscillator allows us to construct a quantum circuit representation of the propagator of the projected Koopman--von Neumann generator.
\end{enumerate}
It is well known that Galerkin projections of transfer operators will in general violate conservation laws or cause numerical artifacts such as spurious oscillations, unless finite-dimensional invariant subspaces are known. We will show that at least for simple linear systems with conserved quantities, such invariant subspaces can be constructed for appropriately defined domains. It is then possible to analytically compute eigenvalues and eigenfunctions and to correctly predict the evolution of probability densities, observables, and wavefunctions contained in the invariant subspace. For more complicated systems, such subspaces can in general not be easily constructed. Choosing suitable projections that preserve the dynamical properties of interest and faithfully simulating such systems on quantum computers remains an open problem \cite{LLASS22}.

While we will focus specifically on deterministic dynamical systems given by autonomous ordinary differential equations, non-deterministic processes governed by stochastic differential equations would be of immense interest as well. Under certain conditions, the Koopman generator associated with a stochastic differential equation can be transformed into a Schrödinger operator and vice versa. This transformation was used in \cite{KNH20, KNP22} to apply data-driven methods developed for the approximation of the Koopman operator or generator to quantum systems. In the same way, we could represent a stochastic process by a Schrödinger equation and apply numerical methods developed for quantum systems. Analyzing relationships between these transformations and the Koopman--von Neumann framework, however, is beyond the scope of this work.

We will introduce the Perron--Frobenius, Koopman, and Koopman--von Neumann generators in Section~\ref{sec:transfer operators}. Furthermore, we will highlight similarities but also major differences between these operators. In Section~\ref{sec:approximation}, we will derive Galerkin projections and in particular data-driven approximations of the Koopman--von Neumann generator. Linear as well as nonlinear benchmark problems will be discussed in Section~\ref{sec:benchmark problems}. Numerical results---using either data-driven approximations or finite element methods---will be presented in Section~\ref{sec:numerical results}. We will conclude with open questions and future work in Section~\ref{sec:conclusion}.

\section{Transfer operators and Koopman--von Neumann mechanics}
\label{sec:transfer operators}

In this section, we will introduce the Koopman--von Neumann (KvN) framework, starting from the Koopman operator description of classical systems. The idea behind Koopman--von Neumann mechanics is to provide quantum representations of classical systems, such that the statistics of the classical dynamics can be recovered by applying the rules of quantum mechanics. In what follows, we will use the notation summarized in Table~\ref{tab:notation}.

\begin{table}
    \centering
    \caption{Operators and their eigenvalues and eigenfunctions.}
    \label{tab:notation}
    \begin{tabular}{l|p{1em}l|c|c}
         & \multicolumn{2}{c|}{differential operator} & eigenvalue & eigenfunction \\ \hline
        Koopman generator & $\mathcal{L} f $&$=\phantom{-}b \vdot \nabla f$ & $\lambda$ & $f_\lambda$ \\
        Perron--Frobenius generator & $\mathcal{L}^* \rho $&$= -b \vdot \nabla \rho - \div(b) \ts \rho $ & $\mu$ & $\rho_\mu$ \\
        Koopman--von Neumann generator & $\mathcal{Q} \ts \psi $&$= -b \vdot \nabla \psi - \frac{1}{2}\div(b) \ts \psi$ & $\nu$ & $\psi_\nu$
    \end{tabular}
\end{table}

\subsection{Koopman and Perron--Frobenius generators}
\label{subsec:koopman_pf_ops}

We begin by providing a high-level description of the Koopman operator framework for general autonomous ordinary differential equations.

\paragraph{Autonomous ordinary differential equations.}

Let $ \Omega \subseteq \R^d $ be the state space. Given an autonomous ordinary differential equation of the form
\begin{equation} \label{eq:ODE}
    \dot{x} = b(x),
\end{equation}
with $ b \colon \Omega \to \R^d $ sufficiently smooth, and an initial condition $ x(0) = x_0 $, the semigroup of Koopman operators $ \{\ts \mathcal{K}^t \ts\} $ is given by
\begin{equation*}
    \big(\mathcal{K}^t f\big)(x) = f\big(\Phi^t(x)\big),
\end{equation*}
where $ \Phi^t $ is the flow map defined by $ \Phi^t\big(x(0)\big) = x(t) $. The infinitesimal generator $ \mathcal{L} $ can then be written as
\begin{equation} \label{eq:Koopman generator}
    \mathcal{L} f = b \vdot \nabla f
                  = \sum_{i=1}^d b_i \ts \pd{f}{x_i}
\end{equation}
and its adjoint $ \mathcal{L}^* $---the generator of the semi-group of Perron--Frobenius operators---as
\begin{equation} \label{eq:Perron-Frobenius generator}
    \mathcal{L}^* \rho = -\div(b \ts \rho) = -\sum_{i=1}^d \pd{(b_i \ts \rho)}{x_i}.
\end{equation}
Using the product rule, we can also write $ \mathcal{L}^* \rho = -b \vdot \nabla \rho - \div(b) \ts \rho $. The operator $ \mathcal{L} $ describes the evolution of observables and $ \mathcal{L}^* $ the evolution of probability densities. For a detailed introduction, see \cite{LaMa94, KNPNCS20}.

\paragraph{Hamiltonian systems.}

As an important illustration, consider a Hamiltonian system defined by a classical Hamiltonian $ H \colon \R^d \times \R^d \to \R $, governed by Hamilton's equation of motion
\begin{equation*}
    \dot{q}_i = \pd{H}{p_i}
    \quad \text{and} \quad
    \dot{p}_i = -\pd{H}{q_i},
\end{equation*}
with $ i = 1, \dots, d $. The vector $ q \in \R^d $ contains the generalized coordinates and $ p \in \R^d $ the corresponding momenta. The Hamiltonian formalism is a convenient mathematical tool to describe conservative mechanical systems~\cite{MHO08}. Such systems---and thus also the associated transfer operators---have many special properties.

Defining $ x = \Big[\begin{smallmatrix} q \\[0.2em] p \end{smallmatrix}\Big] \in \R^{2 \ts d} $, we obtain $ b = \left[\begin{smallmatrix} \phantom{-} \nabla_p H \\ -\nabla_q H \end{smallmatrix}\right] $ and the Koopman generator can be written as
\begin{equation*}
    \mathcal{L} f =
    \begin{bmatrix}[r]
        \nabla_p H \\
        -\nabla_q H
    \end{bmatrix}
    \vdot
    \begin{bmatrix}
        \nabla_q f \\
        \nabla_p f
    \end{bmatrix}
    = \sum_{i=1}^d \left( \pd{H}{p_i} \pd{f}{q_i} - \pd{H}{q_i} \pd{f}{p_i} \right).
\end{equation*}
Similarly, for the Perron--Frobenius generator, we obtain
\begin{equation*}
    \mathcal{L}^* \rho = -\div\left(
    \begin{bmatrix}[r]
        \nabla_p H \\
        -\nabla_q H
    \end{bmatrix}
    \rho \right) = - \sum_{i=1}^d \left( \pd{H}{p_i} \pd{\rho}{q_i} - \pd{H}{q_i} \pd{\rho}{p_i} \right).
\end{equation*}
Note that the second-order derivatives $ \pd{^2 H}{p_i \partial q_i} $ cancel out. It holds that $ \mathcal{L} H = 0 $ and $ \mathcal{L} \exp(-H) = 0 $. Eigenfunctions of $ \mathcal{L} $ associated with the eigenvalue $ \lambda = 0 $ represent conservation laws, see \cite{KKB21} for more details. The crucial observation is that in this case $ \mathcal{L}^* = -\mathcal{L} $, i.e., $ \mathcal{L} $ is a skew-adjoint operator~\cite{LaMa94}. It then follows that $ \mathcal{H} = \mathrm{i} \ts \mathcal{L} $ is a Hermitian operator, which can be used to define a quantum system mirroring the dynamical properties of the classical dynamics. We will show in the next section how to generalize this idea beyond Hamiltonian systems.

\subsection{Koopman--von Neumann mechanics}
\label{subsec:kvn_mechanics}

Instead of considering the evolution of the probability density $ \rho $, Koopman--von Neumann mechanics describes the system's state by a wavefunction~$ \psi $, from which we can compute the probability density $ \rho $ using Born's rule, i.e., $ \rho = \psi^* \psi $. In order to make the introduction of the Koopman--von Neumann framework self-contained, we will include a short derivation. A detailed comparison of the different operators and their properties can be found in \cite{LLASS22}.

\paragraph{Autonomous ordinary differential equations.}

Although Koopman and von Neumann initially focused on Hamiltonian systems, it is possible to derive an evolution equation $ \dot{\psi} = \mathcal{Q} \ts \psi $, the so-called \emph{Koopman--von Neumann equation} (the letter $ \mathcal{Q} $ stands for quantum here), for general ordinary differential equations of the form \eqref{eq:ODE}, where
\begin{equation} \label{eq:KvN generator}
    \mathcal{Q} \ts \psi
        = -b \vdot \nabla \psi - \tfrac{1}{2} \div(b) \ts \psi
        = -\sum_{i=1}^d \left( b_i \pd{\psi}{x_i} + \tfrac{1}{2} \pd{b_i}{x_i} \psi \right),
\end{equation}
see, e.g., \cite{Joseph20, LLASS22}. Then $ \rho := \psi^* \psi $ satisfies the Liouville equation $ \dot{\rho} = \mathcal{L}^* \rho $. This can be shown using the product rule:
\begin{align*}
    \dot{\rho} &= -\sum_{i=1}^d \left( b_i \pd{\psi^*}{x_i} + \tfrac{1}{2} \pd{b_i}{x_i} \psi^* \right) \psi - \psi^* \sum_{i=1}^d \left( b_i \pd{\psi}{x_i} + \tfrac{1}{2} \pd{b_i}{x_i} \psi \right) \\
        &= -\sum_{i=1}^d \left( b_i \left[\pd{\psi^*}{x_i} \psi + \psi^* \pd{\psi}{x_i}\right] + \pd{b_i}{x_i} \psi^* \psi \right) \\
        &= -\sum_{i=1}^d \left( b_i \pd{\rho}{x_i} + \pd{b_i}{x_i} \rho \right) \\
        &= -b \vdot \nabla \rho - \div(b) \ts \rho
         = \mathcal{L}^* \rho.
\end{align*}
Multiplying the Koopman--von Neumann equation by the imaginary unit $ \mathrm{i} $ results in
\begin{equation*}
    \mathrm{i} \ts \dot{\psi} = \mathrm{i} \ts \mathcal{Q} \ts \psi := \mathcal{H} \psi
\end{equation*}
so that the equation closely resembles the time-dependent Schrödinger equation, i.e., $ \mathcal{H} $ is a Hermitian operator and the corresponding propagator is unitary~\cite{LLASS22}.

\begin{remark} \label{rem:skew-symmetric part}
Defining $ \mathcal{L}_s = \tfrac{1}{2} (\mathcal{L} + \mathcal{L}^*) $ and $ \mathcal{L}_a = \tfrac{1}{2} (\mathcal{L} - \mathcal{L}^*) $ to be the symmetric and skew-symmetric parts of the operator $ \mathcal{L} $, we have $ \mathcal{L} = \mathcal{L}_s + \mathcal{L}_a $ and $ \mathcal{L}^* = \mathcal{L}_s - \mathcal{L}_a $. It then follows that
\begin{equation*}
    \mathcal{Q} \ts \psi = -\mathcal{L}_a \psi
        = \tfrac{1}{2}(\mathcal{L}^* - \mathcal{L}) \psi
        = -b \vdot \nabla \psi - \tfrac{1}{2} \div(b) \ts \psi,
\end{equation*}
i.e., the Koopman--von Neumann generator is the adjoint of the skew-symmetric part of the Koopman generator. This decomposition, along with a so-called \emph{warped-phase transformation}, which introduces an additional one-dimensional variable, was also used in the Schrödingerization approach proposed in \cite{JLY24} to solve arbitrary linear partial differential equations on quantum computers. We only require the skew-symmetric part in what follows. A complete Koopman-based quantization procedure, named QECD, for ergodic measure-preserving systems was introduced in \cite{GOPSS22}. As we will show in Appendix~\ref{app:QECD}, the Koopman--von Neumann approach described here can be seen as an extension of QECD to general autonomous ordinary differential equations.
\end{remark}

\paragraph{Hamiltonian systems.}

Since the divergence of a Hamiltonian vector field is zero, we immediately obtain
\begin{equation*}
    \mathcal{Q} = \mathcal{L}^* = -\mathcal{L}.
\end{equation*}
That is, the Koopman generator is in this case already skew-adjoint and the corresponding propagator automatically unitary.

\subsection{Properties of transfer operators}

The spectral properties of the Koopman operator are well-understood, see, e.g., \cite{BMM12}. We include short proofs of the corresponding properties of its infinitesimal generator for the sake of completeness. The question now is how eigenvalues and eigenfunctions of the operators introduced above are related and whether we can apply numerical methods developed for the Koopman generator also to the Koopman--von Neumann generator.

\begin{lemma} \label{lem:properties of operators}
Assuming that the functions constructed below are well-defined, the operators have the following properties:
\begin{enumerate}[leftmargin=4ex, itemsep=0ex, topsep=0.5ex, label=\roman*)]
\item Product rule for the Koopman generator: It holds that $ \mathcal{L}(f \ts g) = (\mathcal{L}f) \ts g + f (\mathcal{L}\ts g) $.
\item Products of eigenfunctions are eigenfunctions: Given two eigenfunctions $ f_{\lambda_1} $ and $ f_{\lambda_2} $ of $ \mathcal{L} $ associated with the eigenvalues $ \lambda_1 $ and $ \lambda_2 $, respectively, it holds that $ \mathcal{L}(f_{\lambda_1} \ts f_{\lambda_2}) = (\lambda_1 + \lambda_2) \ts f_{\lambda_1} \ts f_{\lambda_2} $.
\item Smooth transformations retain conservation laws: Assume that $ \mathcal{L} f_\lambda = \lambda \ts f_\lambda $ and let $ g \colon \R \to \R $ be differentiable, then $ \mathcal{L} (g \circ f_\lambda) = \lambda \ts f_\lambda \cdot g' \circ f_\lambda $. In particular, if $ \lambda = 0 $, then $ \mathcal{L} (g \circ f_\lambda) = 0 $.
\item Powers of eigenfunctions are eigenfunctions: For $ r \in \R $ we have $ f_\lambda^r $ is an eigenfunction associated with the eigenvalue $ r \ts \lambda $.
\item Conservation laws generate additional eigenfunctions of the Perron--Frobenius generator: Let $ \mathcal{L} f_0 = 0 $ and $ \mathcal{L}^* \rho_\mu = \mu \ts \rho_\mu $, then $ \mathcal{L}^*(f_0 \ts \rho_\mu) = \mu \ts (f_0 \ts \rho_\mu) $.
\item Generalized Born rule for the KvN generator: We have $ \mathcal{L}^* (f \ts g) = (\mathcal{Q} \ts f) \ts g + f \ts (\mathcal{Q} \ts g) $.
\item Born rule for eigenfunctions: Given two eigenfunctions $ \psi_{\nu_1} $ and $ \psi_{\nu_2} $ of $ \mathcal{Q} $ corresponding to the eigenvalues $ \nu_1 $ and $ \nu_2 $, we have $ \mathcal{L}^*(\psi_{\nu_1} \ts \psi_{\nu_2}) = (\nu_1 + \nu_2) \ts \psi_{\nu_1} \ts \psi_{\nu_2} $.
\item Eigenfunctions of systems with constant divergence: Assume $ \div(b) \equiv \beta $, where $ \beta \in \R $ is a constant. Given $ \mathcal{L} f_\lambda = \lambda \ts f_\lambda $, we obtain $ \mathcal{L}^* f_\lambda = -(\lambda + \beta) \ts f_\lambda $ and $ \mathcal{Q} f_\lambda = -(\lambda + \tfrac{1}{2} \ts \beta) \ts f_\lambda $.
\end{enumerate}
\end{lemma}

The proofs of these properties can be found in Appendix \ref{app:proof}. Examples of systems with constant divergence are Hamiltonian systems, where $ \beta = 0 $, and linear differential equations of the form $ \dot{x} = B \ts x $, where $ \beta = \tr(B) $. Linear differential equations---including linear Hamiltonian systems---will be discussed in more detail below.

\subsection{Transfer operators on bounded domains}

So far, we only formally introduced the above operators, but did not specify the domain $ \Omega $. Koopman and Perron--Frobenius operators are often defined on $ L^2(\R^d) $, with the requirement that the observables or probability densities decay to zero for $ \norm{x} \to \infty $. In practice, however, we typically have to consider bounded domains $ \Omega $. If the domain is bounded, we assume $ \Omega $ to be forward-invariant under the flow in what follows, i.e., $ \Phi^t(\Omega) \subseteq \Omega $ for all $ t \ge 0 $, so that no boundary conditions are required for the Koopman generator~\cite{MM16}. A key question then is under which conditions the Perron--Frobenius generator retains the same analytical expression as in~\eqref{eq:Perron-Frobenius generator}. The following result shows that imposing homogeneous Dirichlet boundary conditions on probability densities is sufficient. This condition is natural given the forward-invariance of the dynamics.

\begin{lemma}
Given a bounded domain $ \Omega $, let $ \rho $ satisfy homogeneous Dirichlet boundary conditions, i.e., $ \rho(x) = 0 $ on $ \partial \Omega $, then $ \innerprod{\rho}{\mathcal{L} f} = \innerprod{\mathcal{L}^* \rho}{f} $, where $ \mathcal{L}^* \rho = -\div(b \ts \rho) $.
\end{lemma}

\begin{proof}
We follow the proof detailed in \cite{LaMa94}, with the only difference that we consider a bounded domain $ \Omega $. It holds that
\begin{equation*}
    \innerprod{\rho}{\mathcal{L} f} = \int_\Omega \rho \ts \big[b \vdot \nabla f\big] \ts \mathrm{d}x = \int_\Omega \div(b \ts \rho \ts f) \ts \mathrm{d}x - \int_\Omega \div(b \ts \rho) \ts f \ts \mathrm{d}x = \innerprod{\mathcal{L}^* \rho}{f},
\end{equation*}
since, using the divergence theorem and the fact that $ \rho $ is by definition zero on the boundary,
\begin{equation*}
    \int_\Omega \div(b \ts \rho \ts f) \ts \mathrm{d}x = \int_{\partial \Omega} (b \ts \rho \ts f) \cdot \vec{n} \ts \mathrm{d}s = 0,
\end{equation*}
where $ \vec{n} $ is the outward-pointing normal vector.
\end{proof}

This result is important since it allows us to define the Koopman--von Neumann generator $ \mathcal{Q} $ in terms of the Koopman and Perron--Frobenius generators (see Remark~\ref{rem:skew-symmetric part}), now restricted to the space of functions satisfying homogeneous Dirichlet boundary conditions. The formal definition~\eqref{eq:KvN generator} is valid for any smooth function satisfying zero boundary conditions. A complete characterization of the domain of definition for the Koopman--von Neumann generator on bounded domains can be found in~\cite{SGKP24}, where Perron--Frobenius--Sobolev (PFS) spaces are introduced as the natural function space setting for the bounded-domain case. Rather than working only with classical Sobolev spaces, the analysis uses PFS spaces along with an appropriate vanishing-trace condition to incorporate the boundary behavior in a way that is consistent with the Perron--Frobenius framework, with the no-outflow condition providing the boundary-based analogue of forward invariance by ensuring that trajectories do not leave the domain under the forward flow.

\subsection{Linear dynamical systems and invariant subspaces}

In order to study the Koopman--von Neumann generator and its properties in more depth, we first consider linear dynamical systems since a wealth of their properties can be computed explicitly, see, e.g.,~\cite{BMM12, KKS16}.

\begin{lemma} \label{lem:linear ODE}
Given $ \dot{x} = B \ts x $, where $ B \in \R^{d \times d} $, let $ w \in \C^d $ be an eigenvector of $ B^\top $ corresponding to the eigenvalue $ \lambda \in \C $, i.e., $ B^\top w = \lambda \ts w $. Then $ f_\lambda(x) = x \vdot w $ is an eigenfunction of the Koopman generator.\!\footnote{We follow the physics convention and define the inner product to be conjugate-linear in the first variable.}
\end{lemma}

\begin{proof}
We have $ \mathcal{L} f_\lambda(x) = B \ts x \vdot w = x \vdot B^\top w = x \vdot (\lambda \ts w) = \lambda \ts f_\lambda(x) $.
\end{proof}

In what follows, we will call the eigenfunctions given by the eigenvectors of $ B^\top $ \emph{principal eigenfunctions}. Note that these functions are not bounded if we consider $ \Omega = \R^d $ and do in general not satisfy the boundary conditions if the domain $ \Omega $ is bounded. We can now construct infinitely many additional eigenfunctions by considering products of these eigenfunctions as shown in Lemma~\ref{lem:properties of operators}. As a special case, we will analyze Hamiltonian systems.

\begin{definition}
Let $ J \in \R^{2 \ts d \times 2 \ts d} $ be the matrix given by
\begin{equation*}
    J =
    \begin{bmatrix}
        0 & I\,\, \\
        -I & 0\,\,
    \end{bmatrix},
\end{equation*}
then $ B \in \R^{2 \ts d \times 2 \ts d} $ is called \emph{Hamiltonian} if $ J \ts B + B^\top J = 0 $. Note that $ J^{-1} = J^\top = -J $.
\end{definition}

Writing $ B $ as a block matrix
\begin{equation*}
    B =
    \begin{bmatrix}
        B_{11} & B_{12} \\
        B_{21} & B_{22}
    \end{bmatrix},
\end{equation*}
it is Hamiltonian if $ B_{11}^\top + B_{22} = 0 $ and $ B_{12} $ and $ B_{21} $ are symmetric \cite{MHO08}. It then follows that $ \dot{x} = B \ts x $ is a Hamiltonian system with
\begin{equation*}
    H(q, p) = - \tfrac{1}{2} \ts q^\top B_{21} \ts q + p^\top B_{11} q + \tfrac{1}{2} \ts p^\top B_{12} \ts p.
\end{equation*}
The spectrum of Hamiltonian matrices is symmetric about $ 0 $, i.e., $ \lambda \in \sigma(B) \!\!\implies\!\! -\lambda \in \sigma(B) $. If all eigenvalues of $ B $ lie on the imaginary axis, then solutions are purely oscillatory.

\begin{example} \label{ex:oscillator}
A \emph{damped oscillator} is given by $ \ddot{q} + \gamma \ts \dot{q} + \omega^2 \ts q = 0 $, where $ \gamma $ is the damping ratio and $ \omega $ the angular frequency. Defining $ x = \left[\begin{smallmatrix} q \\ p \end{smallmatrix}\right] $, this can be rewritten as
\begin{equation*}
    \dot{x} =
    \underbrace{\begin{bmatrix}
        0 & 1 \\
        -\omega^2 & -\gamma
    \end{bmatrix}}_{=: B}
    x.
\end{equation*}
We obtain the \emph{undamped oscillator} by setting $ \gamma = 0 $. The Hamiltonian of the system is in this case given by $ H(q, p) = \frac{1}{2} \ts \omega^2 \ts q^2 + \frac{1}{2} \ts p^2 $. \exampleSymbol
\end{example}

We will now construct invariant subspaces for linear systems, which enable us to define suitable dictionaries for the numerical approximation techniques derived in Section~\ref{sec:approximation}. The first result shows that finite-dimensional subspaces spanned by monomials up to a fixed order $ r $ are invariant under the action of the Koopman generator.

\begin{lemma} \label{lem:invariant subspaces 1}
Let $ \alpha = (\alpha_1, \dots, \alpha_d) \in \mathbb{N}_0^d $ be a multi-index with $ \abs{\alpha} = \sum_{i=1}^d \alpha_i $ and define
\begin{equation*}
    \mathbb{V} = \mspan\big\{ x^\alpha = x_1^{\alpha_1} \cdots x_d^{\alpha_d} : \abs{\alpha} \le r\big\}
\end{equation*}
to be the finite-dimensional subspace spanned by monomials of order up to $ r $. Given a dynamical system of the form $ \dot{x} = B \ts x $ and a function $ f \in \mathbb{V} $, it holds that $ \mathcal{L} f \in \mathbb{V} $.
\end{lemma}

\begin{proof}
Let $ e_i \in \mathbb{N}_0^d $ denote the $ i $th unit vector. First, assume that $ f(x) = x^\alpha $, with $ \abs{\alpha} = q \le r $, then
\begin{equation*}
    \pd{f}{x_i} =
    \begin{cases}
        \alpha_i \ts x^{\alpha - e_i}, & \alpha_i \ne 0, \\
        0, & \text{otherwise},
    \end{cases}
\end{equation*}
is either a monomial of order $ q-1 $ or identical to zero. Thus,
\begin{equation*}
    \mathcal{L} f = B \ts x \vdot \nabla f = \sum_{i=1}^d \sum_{j=1}^d B_{ij} \pd{f}{x_i} \ts x_j
\end{equation*}
is at most a polynomial of degree $ q \le r $ and thus contained in $ \mathbb{V} $. For an arbitrary function $ f \in \mathbb{V} $, we apply the same logic to each monomial term.
\end{proof}

We can now use this results to construct basis functions that automatically satisfy the boundary conditions if the domain is bounded.

\begin{lemma} \label{lem:invariant subspaces 2}
Assume that there exists a conservation law $ f_0 $, i.e., $ \mathcal{L} f_0 = 0 $, that vanishes on the boundary, that is, $ f_0(x) = 0 $ for all $ x \in \partial \Omega $. We define a new function space
\begin{equation*}
    \mathbb{V}' = \mspan\big\{ f_0 \ts f : f \in \mathbb{V} \}
\end{equation*}
so that all functions satisfy the boundary conditions. Given $ f' \in \mathbb{V}' $, it holds that $ \mathcal{L} f' \in \mathbb{V}' $.
\end{lemma}

\begin{proof}
Using Lemma~\ref{lem:properties of operators} and Lemma~\ref{lem:invariant subspaces 1}, we have
\begin{equation*}
    \mathcal{L}(f_0 \ts f) = \underbrace{(\mathcal{L}f_0)}_{=0} f + f_0 \underbrace{(\mathcal{L}\ts f)}_{\in \mathbb{V}} \in \mathbb{V}'. \qedhere
\end{equation*}
\end{proof}

The same result also applies to the Perron--Frobenius and Koopman--von Neumann generators as the following corollary shows.

\begin{corollary}
The function space $ \mathbb{V}' $ is also invariant under the action of the Perron--Frobenius generator $ \mathcal{L}^* $ and Koopman--von Neumann generator $ \mathcal{Q} $.
\end{corollary}

\begin{proof}
Let $ \psi \in \mathbb{V}' $, then
\begin{equation*}
    \mathcal{Q} \psi = -\underbrace{b \vdot \nabla \psi}_{= \mathcal{L} \psi \in \mathbb{V}'} \;- \; \tfrac{1}{2} \underbrace{\div(b) \ts \psi}_{\mathclap{=\tr(B) \psi \in \mathbb{V}'}} \in \mathbb{V}'.
\end{equation*}
The proof for $ \mathcal{L}^* $ follows in the same way.
\end{proof}

Given a conservation law $ f_0 $, this allows us to define a suitable domain by choosing the boundary $ \partial \Omega $ to be a level set of the function $ f_0 $, provided it generates a valid bounded domain. Invariant subspaces for the operators $ \mathcal{L} $, $\mathcal{L}^*$, and $ \mathcal{Q} $ can then be constructed as we will illustrate in Section~\ref{sec:benchmark problems}.

\section{Approximation of the Koopman--von Neumann generator}
\label{sec:approximation}

Numerical methods for the approximation of the Koopman operator and its infinitesimal generator gained a lot of attention in recent years, see, e.g., \cite{WKR15, KNPNCS20}. We can now exploit the relationships between the operators derived in Section~\ref{sec:transfer operators} to estimate the Koopman--von Neumann generator.

\subsection{Galerkin projection}

Let $ \{ \phi_i \}_{i=1}^n $ be a set of linearly independent basis functions that, if the domain is bounded, satisfy the homogeneous Dirichlet boundary conditions. Define $ \mathbb{V} = \mspan\{ \phi_i \}_{i=1}^n $ to be the generated $ n $-dimensional subspace and
\begin{equation*}
    \phi(x) = [\phi_1(x), \dots, \phi_n(x)]^\top \in \R^n.
\end{equation*}
We can write any function $ f \in \mathbb{V} $ as a linear combination of the basis functions, i.e.,
\begin{equation*}
    f(x) = \sum_{i=1}^n c_i \ts \phi_i(x) = c^\top \phi(x),
\end{equation*}
where $ c = [c_1, \dots, c_n]^\top \in \R^n $. Let $ \mathcal{P} $ denote the projection onto $ \mathbb{V} $, then the matrix representation $ L \in \R^{n \times n} $ of the Galerkin approximation $ \restr{\mathcal{L}}{\mathbb{V}} := \mathcal{P} \ts \mathcal{L} \ts \mathcal{P} $ of the Koopman generator $ \mathcal{L} $ is given by
\begin{equation*}
    L = G^{-1} A,
\end{equation*}
where
\begin{equation*}
    G_{ij} = \innerprod{\phi_i}{\phi_j}
    \quad \text{and} \quad
    A_{ij} = \innerprod{\phi_i}{\mathcal{L} \phi_j}.
\end{equation*}
We can also write this in compact form as
\begin{equation*}
    G = \int \! \phi(x) \ts \phi(x)^\top \ts \mathrm{d}x
    \quad \text{and} \quad
    A = \int \! \phi(x) \ts \mathcal{L}\phi(x)^\top \ts \mathrm{d}x,
\end{equation*}
where $ \mathcal{L} $ is applied component-wise. For a function $ f(x) = c^\top \phi(x) $, we then obtain
\begin{equation*}
    \restr{\mathcal{L}}{\mathbb{V}} \ts f(x) = (L \ts c)^\top \phi(x).
\end{equation*}
Since $ \innerprod{\phi_i}{\mathcal{L}^* \phi_j} = \innerprod{\mathcal{L} \phi_i}{\phi_j} = A_{ji} $, the matrix representation $ L^* $ of the projected Perron--Frobenius generator  $ \restr{\mathcal{L}^*}{\mathbb{V}} $ is given by
\begin{equation*}
    L^* = G^{-1} A^\top.
\end{equation*}
Similarly, in order to approximate the projected Koopman--von Neumann generator $ \restr{\mathcal{Q}}{\mathbb{V}} $, we can compute $ \innerprod{\phi_i}{\mathcal{Q} \ts \phi_j} = \frac{1}{2} \big(\!\innerprod{\phi_i}{\mathcal{L}^* \phi_j} - \innerprod{\phi_i}{\mathcal{L} \phi_j}\!\big) = \frac{1}{2} (A_{ji} - A_{ij}) $ so that its matrix representation is given by
\begin{equation*}
    Q = \tfrac{1}{2} \ts G^{-1} \big(A^\top - A\big).
\end{equation*}
The matrix $ Q $ is now skew-symmetric w.r.t.\ the $ G $-weighted inner product $ \innerprod{\alpha}{\beta}_G = \alpha^\top G \ts \overline{\beta} $ but not the standard Euclidean inner product $ \innerprod{\alpha}{\beta} = \alpha^\top \ts \overline{\beta} $. We can enforce skew-symmetry w.r.t.\ the unweighted inner product by applying a so-called whitening transformation. Since the basis functions are by assumption linearly independent, the matrix $ G $ is symmetric and positive definite so that its eigenvalues are positive. Let $ G = V \ts D \ts V^\top $ be the eigendecomposition of $ G $ and define $ \widetilde{\phi}(x) = D^{-\nicefrac{1}{2}} \ts V^\top \ts \phi(x) $, then
\begin{equation*}
    \widetilde{G} = \int \! \widetilde{\phi}(x) \ts \widetilde{\phi}(x)^\top \ts \mathrm{d}x = \int \! D^{-\nicefrac{1}{2}} V^\top \ts \phi(x) \ts \phi(x)^\top V D^{-\nicefrac{1}{2}} \ts \mathrm{d}x = D^{-\nicefrac{1}{2}} V^\top \ts G \ts V D^{-\nicefrac{1}{2}} = I
\end{equation*}
and
\begin{equation*}
    \widetilde{A} = \int \! \widetilde{\phi}(x) \ts \mathcal{L}\widetilde{\phi}(x)^\top \ts \mathrm{d}x = \int \! D^{-\nicefrac{1}{2}} V^\top \ts \phi(x) \ts \mathcal{L}\phi(x)^\top V D^{-\nicefrac{1}{2}} \ts \mathrm{d}x = D^{-\nicefrac{1}{2}} V^\top \ts A \ts V D^{-\nicefrac{1}{2}},
\end{equation*}
i.e., the transformed basis functions $ \widetilde{\phi}_i(x) $ are orthonormal and $ \widetilde{Q} = \tfrac{1}{2} \ts \widetilde{G}^{-1} \big(\widetilde{A}^\top - \widetilde{A}\big) = \tfrac{1}{2} \big(\widetilde{A}^\top - \widetilde{A}\big) $. This shows that it is always possible to define the basis functions in such a way that the matrix representation of the projected Koopman--von Neumann generator is skew-symmetric w.r.t.\ the standard Euclidean inner product so that the corresponding propagator $ e^{t\widetilde{Q}} $ is unitary.

In order to compute eigenvalues and eigenfunctions of the projected Koopman--von Neumann generator $ \restr{\mathcal{Q}}{\mathbb{V}} $, we have to compute eigenvalues and eigenvectors of $ Q $. Given $ Q \ts v = \nu \ts v $, this implies that $ \psi_\nu(x) = v^\top \phi(x) $ is an eigenfunction since
\begin{equation*}
    \restr{\mathcal{Q}}{\mathbb{V}} \ts \psi_\nu(x) = (Q \ts v)^\top \phi(x) = \nu \ts v^\top \phi(x) = \nu \ts \psi_\nu(x).
\end{equation*}
Additionally, the temporal evolution of the coefficients $ c $ of a function $ \psi(x) = c^\top \phi(x) $ can now be described by the system of linear ordinary differential equations $ \dot{c} = Q \ts c $. This allows us to propagate projected wavefunctions in time.

\subsection{Data-driven approximation}

A data-driven approach to approximate the Koopman generator called \emph{generator extended dynamic mode decomposition} (gEDMD), which can be viewed as a Galerkin approximation, where the matrices $ A $ and $ G $ are estimated using Monte Carlo integration, was proposed in \cite{KNPNCS20}. Our goal is to extend this method to the Koopman--von Neumann generator. Given uniformly sampled training data $ x^{(l)} $, with $ l = 1, \dots, m $, and the corresponding time derivatives $ \dot{x}^{(l)} $, which can, for instance, be estimated using finite difference approximations, we define
\begin{equation*}
    \dot{\phi}_k(x) = (\mathcal{L} \phi_k)(x) = \sum_{i=1}^d b_i(x) \ts \pd{\phi_k}{x_i}(x)
\end{equation*}
and compute the matrices
\begin{equation*}
    \Phi_X =
    \begin{bmatrix}
        \phi_1(x_1) & \dots  & \phi_1(x_m) \\
        \vdots      & \ddots & \vdots      \\
        \phi_n(x_1) & \dots  & \phi_n(x_m)
    \end{bmatrix}
    \quad \text{and} \quad
    \dot{\Phi}_X =
    \begin{bmatrix}
        \dot{\phi}_1(x_1) & \dots  & \dot{\phi}_1(x_m) \\
        \vdots            & \ddots & \vdots            \\
        \dot{\phi}_n(x_1) & \dots  & \dot{\phi}_n(x_m)
    \end{bmatrix}.
\end{equation*}
The required derivatives of the basis functions can, for example, be computed using automatic differentiation. We then have
\begin{alignat*}{4}
    \widehat{G} &:= \tfrac{1}{m} \Phi_X \ts \Phi_X^\top = \tfrac{1}{m} \sum_{l=1}^m \phi(x^{(l)}) \ts \phi(x^{(l)})^\top &&\underset{\scriptscriptstyle m \rightarrow \infty}{\longrightarrow} \int \phi(x) \ts \phi(x)^\top \ts \mathrm{d}x &&= G, \\
    \widehat{A} &:= \tfrac{1}{m} \Phi_X \ts \dot{\Phi}_X^\top = \tfrac{1}{m} \sum_{l=1}^m \phi(x^{(l)}) \ts \dot{\phi}(x^{(l)})^\top &&\underset{\scriptscriptstyle m \rightarrow \infty}{\longrightarrow} \int \phi(x) \ts \mathcal{L} \phi(x)^\top \ts \mathrm{d}x &&= A,
\end{alignat*}
so that $ \widehat{L} = \widehat{G}^+ \widehat{A} $ is a consistent approximation of the projected Koopman generator, where $ ^+ $ denotes the pseudoinverse, see \cite{KNPNCS20}. Similarly, we can compute empirical estimates of the projected Perron--Frobenius generator and Koopman--von Neumann generator using
\begin{equation*}
    \widehat{L}^* = \widehat{G}^+ \widehat{A}^{\ts\top}
    \quad \text{and} \quad
    \widehat{Q} = \tfrac{1}{2} \ts \widehat{G}^+ \big(\widehat{A}^{\ts\top} - \widehat{A}\big),
\end{equation*}
respectively. Using whitened basis functions as described above, we then obtain a skew-symmetric matrix representation of the Koopman--von Neumann generator computed from data.

\begin{remark} We would like to point out that:
\begin{enumerate}[leftmargin=4ex, itemsep=0ex, topsep=0.5ex, label=\roman*)]
\item In the derivation above, we implicitly assumed that the data points $ x^{(l)} $ are uniformly distributed. If we estimate the operators from trajectory data, this is in general not the case and we compute a Galerkin projection w.r.t.\ a weighted inner product, see \cite{WKR15, KKS16, KM18, BF23} for a more detailed analysis.
\item A data-driven method for estimating the Koopman generator that does not require time-derivatives is described in \cite{MauGon16}. The approximation of the generator is in this case obtained by taking the logarithm of a matrix representation of the Koopman operator estimated from time series data. The method could hence also be used to approximate the Koopman--von Neumann generator.
\end{enumerate}
\end{remark}

\section{Benchmark problems}
\label{sec:benchmark problems}

In this section, we will introduce and analyze benchmark problems, namely simple undamped and damped oscillators and the Lotka--Volterra model. We will in particular illustrate some of the similarities and structural differences between the Koopman and Koopman--von Neumann generators that occur for different systems and boundary conditions. Numerical results for these systems will be presented in Section~\ref{sec:numerical results}.

\subsection{Undamped oscillator}

We first consider the undamped oscillator defined in Example~\ref{ex:oscillator} and choose $ \omega = \sqrt{2} $ so that
\begin{equation*}
    B =
    \begin{bmatrix}
        0 & 1 \\
        -2 & 0
    \end{bmatrix}.
\end{equation*}
Using Lemma~\ref{lem:linear ODE}, the principal eigenvalues and eigenfunctions of the Koopman generator are
\begin{alignat*}{2}
    \lambda_1 &= \phantom{+} \mathrm{i} \ts \sqrt{2}, &\qquad  f_{\lambda_1}(x) &= \phantom{+} \mathrm{i} \ts x_1 + \tfrac{1}{\sqrt{2}} x_2, \\
    \lambda_2 &= -\mathrm{i} \ts \sqrt{2}, & f_{\lambda_2}(x) &= -\mathrm{i} \ts x_1 + \tfrac{1}{\sqrt{2}} x_2,
\end{alignat*}
with period $ T = \frac{2 \ts \pi}{\sqrt{2}} \approx 4.44 $. Additionally, any function of the form $ f_\lambda(x) = f_{\lambda_1}(x)^{k_1} f_{\lambda_2}(x)^{k_2} $ is an eigenfunction corresponding to the eigenvalue $ \lambda = k_1 \ts \lambda_1 + k_2 \ts \lambda_2 = \mathrm{i} \ts \sqrt{2}(k_1 - k_2) $. All these eigenvalues lie on the imaginary axis. Note that if $ k_1 = k_2 $, then $ \lambda = 0 $. In particular, for $ k_1 = k_2 = 1 $, we obtain $ f_\lambda(x) = f_{\lambda_1}(x) \ts f_{\lambda_2}(x) = x_1^2 + \frac{1}{2} x_2^2 = H(q, p) =: H(x) $, i.e., the Hamiltonian itself. We now show how to construct eigenfunctions satisfying boundary conditions in different settings:
\begin{enumerate}[leftmargin=4ex, itemsep=0ex, topsep=0.5ex, label=\roman*)]
\item Let $ \Omega = \R^2 $, then the Gaussian-like function
\begin{equation*}
    f_0(x) = \sum_{k=0}^\infty \frac{\big(\!-\!H(x)\big)^k}{k!} = \exp\big(\!-\!x_1^2 - \tfrac{1}{2} x_2^2\big)
\end{equation*}
is a conservation law, which respects zero boundary conditions at infinity. In order to construct additional eigenfunctions, we can multiply $ f_0 $ by powers of $ f_{\lambda_1} $ and $ f_{\lambda_2} $. These functions are bounded and tend to zero for $ \norm{x} \to \infty $.
\item As an example with a bounded domain, consider $ \Omega = \big\{ x \in \R^2: x_1^2 + \frac{1}{2} \ts x_2^2 < 1 \big\} $, i.e., an ellipse with semi-minor axis $ 1 $ and semi-major axis $ \sqrt{2} $. Then
\begin{equation}
\label{eq:conservation_law_undamped}
    f_0(x) = 1 - x_1^2 - \tfrac{1}{2} x_2^2
\end{equation}
is a conservation law that vanishes on the boundary (i.e., Lemma~\ref{lem:invariant subspaces 2} applies) and can again be multiplied by powers of $ f_{\lambda_1} $ and $ f_{\lambda_2} $ to obtain additional eigenfunctions that satisfy the boundary conditions.
\end{enumerate}

\begin{figure}
    \centering
    \begin{minipage}[t]{0.35\linewidth}
        \centering
        \subfiguretitle{(a)}
        \vspace*{1ex}
        \includegraphics[width=\linewidth]{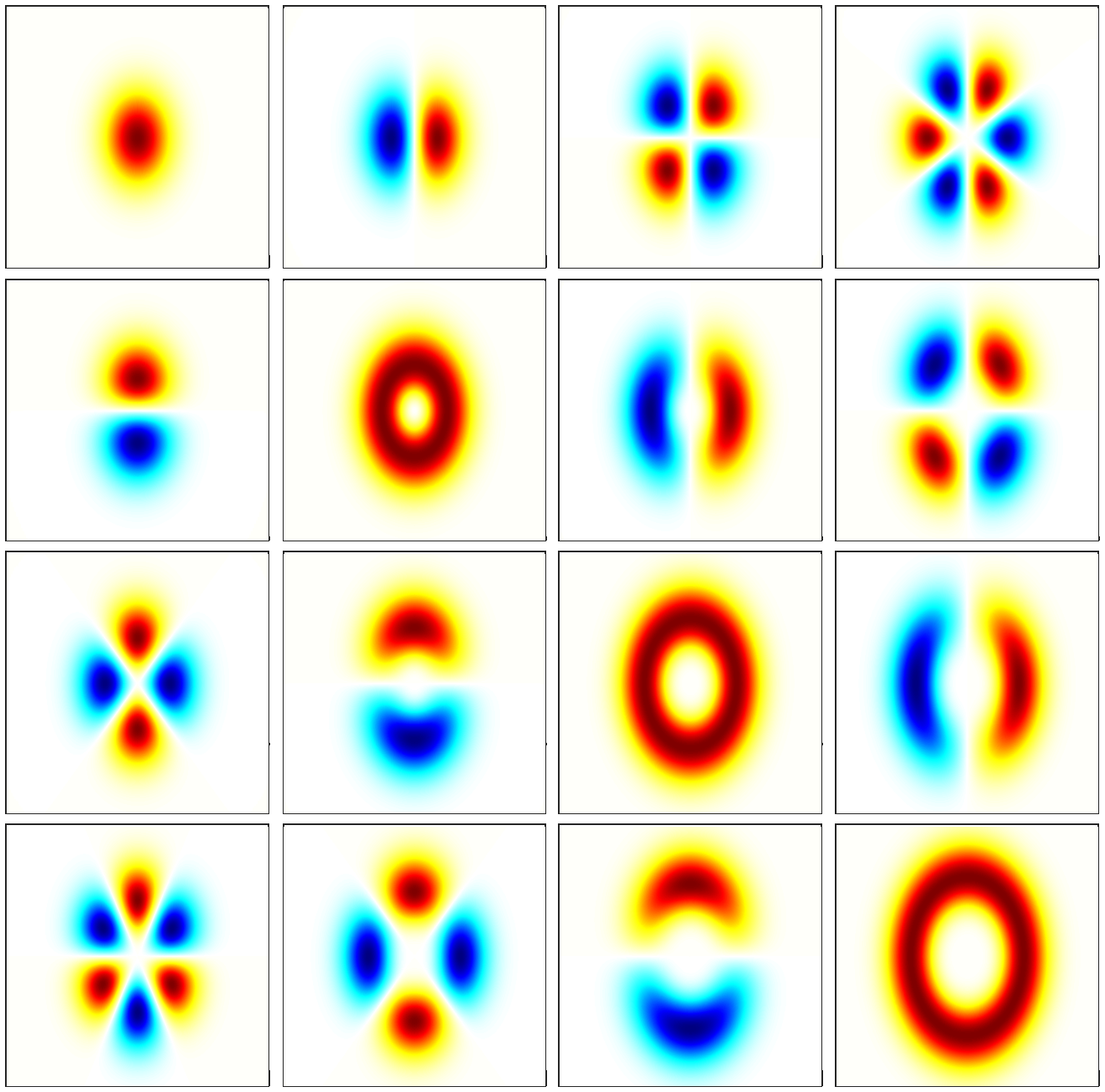}
    \end{minipage}
    \hspace*{6ex}
    \begin{minipage}[t]{0.327\linewidth}
        \centering
        \subfiguretitle{(b)}
        \vspace*{1ex}
        \includegraphics[width=\linewidth]{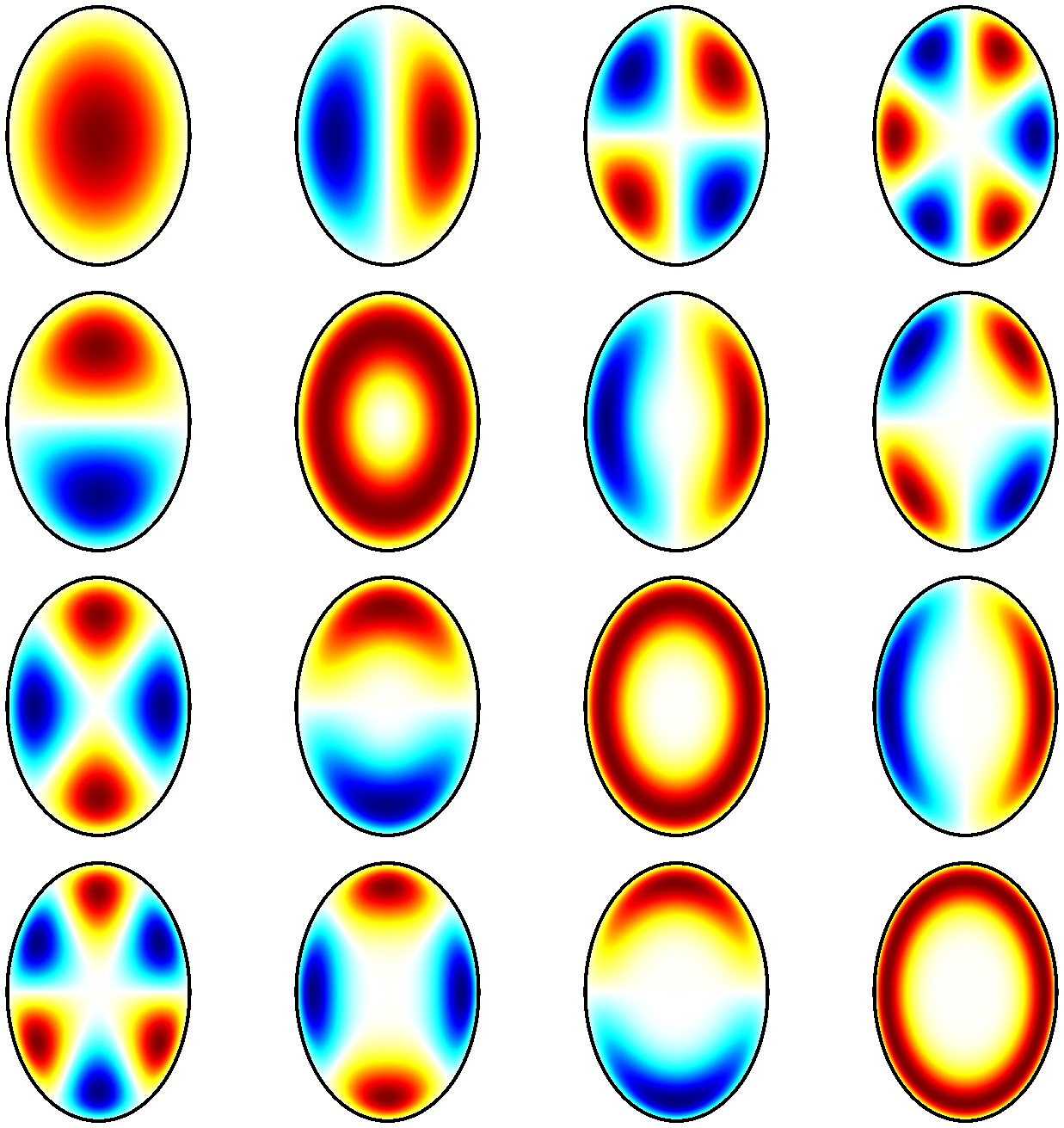}
    \end{minipage}
    \caption{Real part (if $ k_1 \le k_2 $) or imaginary part (if $ k_1 > k_2 $ since the real part would be the same as that of the function with $ k_1 $ and $ k_2 $ swapped) of the eigenfunctions $ f_\lambda(x) = f_0(x) \ts f_{\lambda_1}(x)^{k_1} \ts f_{\lambda_2}(x)^{k_2} $ of the Koopman generator (and hence also Koopman--von Neumann generator) associated with the undamped oscillator defined on the domains (a)~$ \Omega = \R^2 $ and (b)~$ \Omega = \big\{ x \in \R^2: x_1^2 + \frac{1}{2} \ts x_2^2 < 1 \big\} $, where $ k_1, k_2 \in \{0, 1, 2, 3\} $ are the row and column numbers, respectively. The eigenfunctions are plotted for (a) $ x \in [-4, 4]^2 $ and (b)~$ x \in \Omega $. Red represents positive and blue negative values. The corresponding eigenvalues are $ \lambda = \mathrm{i} \ts \sqrt{2}(k_1 - k_2) $, i.e., multiples of the angular frequency~$ \omega $.  All eigenvalues are as expected purely imaginary.}
    \label{fig:UO eigenfunctions}
\end{figure}

A few eigenfunctions of the Koopman generator are shown in Figure~\ref{fig:UO eigenfunctions}. Since $ \mathcal{L}^* = \mathcal{Q} = -\mathcal{L} $ here and the constructed eigenfunctions for the bounded domain case satisfy the boundary conditions, these functions are also eigenfunctions of the Perron--Frobenius generator and Koopman--von Neumann generator, associated with the eigenvalues $ \mu = \nu = -\lambda $.

\subsection{Damped oscillator}

Let us now consider the damped case. We again choose $ \omega = \sqrt{2} $, but set $ \gamma = 2 $ so that
\begin{equation*}
    B =
    \begin{bmatrix}
        0 & 1 \\
        -2 & -2
    \end{bmatrix}.
\end{equation*}
Using Lemma~\ref{lem:linear ODE}, the principal eigenvalues and eigenfunctions of the Koopman generator are
\begin{alignat*}{2}
    \lambda_1 &= -1 + \mathrm{i}, &\qquad f_{\lambda_1}(x) &= x_1 + \tfrac{1}{2} (1 - \mathrm{i}) \ts x_2, \\
    \lambda_2 &= -1 - \mathrm{i}, &       f_{\lambda_2}(x) &= x_1 + \tfrac{1}{2} (1 + \mathrm{i}) \ts x_2.
\end{alignat*}
By taking integer powers of these eigenfunctions, the grid of eigenvalues $ \lambda = -(k_1 + k_2) + \mathrm{i} \ts (k_1 - k_2)$ and eigenfunctions shown in Figure~\ref{fig:DO eigenfunctions} can be generated. These eigenfunctions again grow to infinity with increasing $ \norm{x} $. Unlike for the undamped oscillator, we cannot easily find a conservation law in this case. We can, however, find a bounded domain $ \Omega = \big\{ x \in \R^2: x_1^2 + x_1 \ts x_2 + \frac{1}{2} \ts x_2^2 < 1 \big\} $ that satisfies the forward-invariance condition. The boundary of this domain is a level set of the eigenfunction for $ k_1 = k_2 = 1 $, i.e., $ f_\lambda = f_{\lambda_1}(x) \ts f_{\lambda_2}(x) = x_1^2 + x_1 \ts x_2 + \frac{1}{2} x_2^2 $, with eigenvalue $ \lambda = - 2 $. The function $ f_\lambda $ is not a conservation law, and hence we cannot simply multiply the Koopman generator eigenfunctions by $ f_\lambda $ to obtain eigenfunctions of $ \mathcal{L}^* $ or $ \mathcal{Q} $. In fact, an eigenfunction of the Perron--Frobenius operator would be a Dirac distribution centered at the origin. However, such eigenfunctions capture only local information about the dynamics~\cite{MM16}. This example illustrates some of the structural differences of the spectra of $ \mathcal{L} $ and $ \mathcal{Q} $ that can occur.

Nevertheless, we can still use the Perron--Frobenius and Koopman--von Neumann operators or generators to propagate probability densities or wavefunctions. This will be illustrated in Section~\ref{sec:numerical results}.

\begin{figure}
    \centering
    \begin{minipage}[t]{0.36\linewidth}
        \centering
        \subfiguretitle{(a)}
        \vspace*{2ex}
        \includegraphics[width=\linewidth]{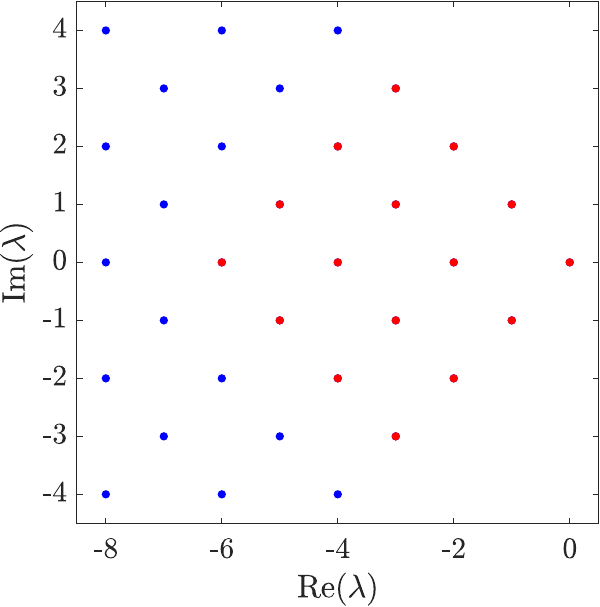}
    \end{minipage}
    \hspace*{6ex}
    \begin{minipage}[t]{0.327\linewidth}
        \centering
        \subfiguretitle{(b)}
        \vspace*{1ex}
        \includegraphics[width=\linewidth]{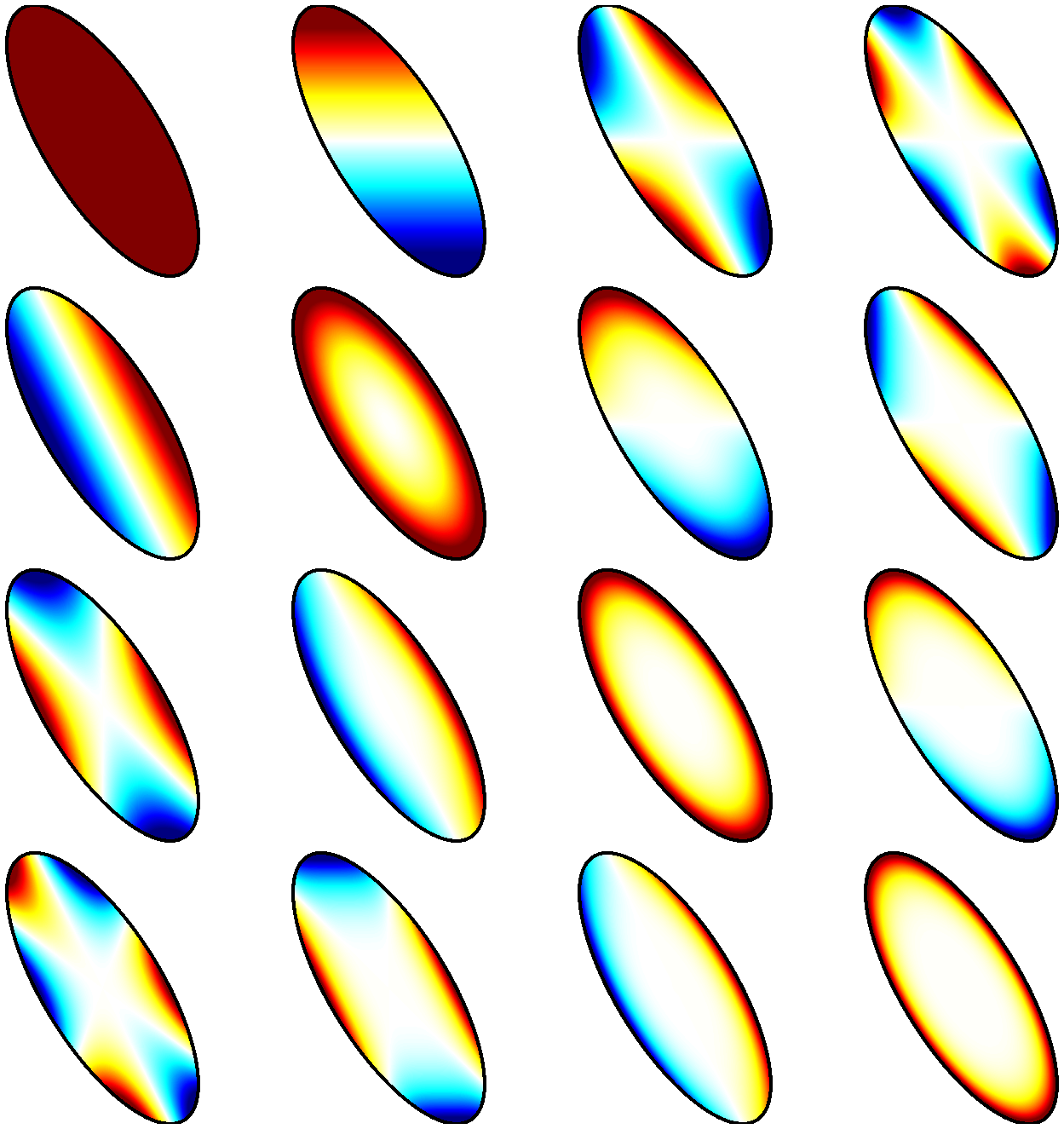}
    \end{minipage}
    \caption{(a) Lattice structure generated by the principal eigenvalues of the Koopman generator associated with the damped oscillator. The eigenvalues marked in red correspond to the eigenfunctions shown on the right. (b) Real or imaginary parts of the eigenfunctions $ f_\lambda(x) = f_{\lambda_1}(x)^{k_1} \ts f_{\lambda_2}(x)^{k_2} $ of the Koopman generator defined on $ \Omega = \big\{ x \in \R^2: x_1^2 + x_1 \ts x_2 + \frac{1}{2} \ts x_2^2 < 1 \big\} $. Note that the eigenfunctions do not vanish on the boundary. The corresponding eigenvalues are $ \lambda = -(k_1 + k_2) + \mathrm{i} \ts (k_1 - k_2) $, where $ k_1, k_2 \in \{0, 1, 2, 3\} $ are the row and column numbers, respectively.}
    \label{fig:DO eigenfunctions}
\end{figure}

\begin{remark}
It was shown in \cite{MM16} that eigenfunctions of the Koopman operator or generator can be used to analyze the stability of dynamical systems. Indeed, the domain $ \Omega $ chosen for the damped oscillator is forward-invariant and the origin $ x^* = 0 $ is a stable fixed point. A Lyapunov function is then an observable $ V(x) > 0 $ for $ x \ne x^* $ that satisfies $ \mathcal{L} V(x) < 0 $ for all $ x \ne x^* $. The eigenfunction for $ k_1 = k_2 = 1 $ with eigenvalue $ \lambda = - 2 $ has exactly this property.
\end{remark}

\subsection{Lotka--Volterra model}

As a prototypical nonlinear dynamical system, we choose the Lotka--Volterra model $ \dot{x} = b(x) $, with
\begin{equation*}
    b(x) =
    \begin{bmatrix}
        x_1 \ts (1 - x_2) \\
        x_2 \ts (x_1 - 1)
    \end{bmatrix}
\end{equation*}
and $ \div(b) = x_1 - x_2 $, which describes the dynamics of two interacting species, one being a predator and the other its prey. Both the predator and prey populations oscillate around a fixed point. Unlike for the undamped oscillator, the period length now depends on an associated ``energy'' \cite{Waldvogel86}. The Lotka--Volterra model can be turned into a Hamiltonian system by a nonlinear coordinate transformation, but in its original form the dynamics are not Hamiltonian. A well-known conserved quantity of the system is
\begin{equation}
\label{eq:conservation_lv}
    f_0(x) = x_1 + x_2 - \log(x_1) - \log(x_2).
\end{equation}
We define the domain $ \Omega = \big\{ x \in \R^2 : x_1 + x_2 - \log(x_1) - \log(x_2) < 3 \big\} $ by a level set of $ f_0 $, or, equivalently, a closed orbit of the system. It was shown in \cite{MQ15} that $ \rho_0(x) = \frac{1}{x_1 \ts x_2} $ is an eigenfunction of the Perron--Frobenius generator with eigenvalue $ \mu = 0 $ and thus, using Lemma~\ref{lem:properties of operators}, also the function
\begin{equation*}
    \widetilde{\rho}_0(x) = \big(3 - f_0(x)\big) \rho_0(x),
\end{equation*}
which now satisfies the boundary conditions. This (unnormalized) density is invariant under the dynamics, which implies that its square-root is an invariant wavefunction.

\section{Numerical results}
\label{sec:numerical results}

We will now approximate the Koopman--von Neumann generator using data-driven approaches as well as finite element methods. We then use the discretized generator to compute eigenvalues and eigenfunctions or to propagate wavefunctions.

\subsection{Undamped oscillator}

Let us first consider the undamped oscillator on the bounded domain $ \Omega $ defined in Section~\ref{sec:benchmark problems}. For this example, we will construct both an analytical discretization using a polynomial basis set and a generic approximation using random features, and then study their approximation properties. In addition, we use the analytical discretization to construct a quantum circuit representation for the Koopman--von Neumann operator $ e^{t\tilde{Q}} $.

\paragraph{Monomial basis.}

First, based on Lemma~\ref{lem:invariant subspaces 2}, we select the basis functions
\begin{equation*}
    \phi(x) = (1 - x_1^2 - \tfrac{1}{2} x_2^2)
    \begin{bmatrix}
        1 \\
        x_1 \\
        x_2 \\
        x_1^2 \\
        x_1 \ts x_2 \\
        x_2^2
    \end{bmatrix},
\end{equation*}
that is, polynomial functions up to degree two multiplied by the conservation law $ f_0 $ from~\eqref{eq:conservation_law_undamped}. These basis functions exactly satisfy the Dirichlet boundary conditions on the domain and, as we have seen, they also define an invariant subspace for the Koopman--von Neumann generator. The Galerkin matrices $ G $ and $ A $ can in this case be computed analytically. We have
\begin{equation*}
    G = \sqrt{2} \ts \pi
    \begin{bmatrix}
         \frac{1}{3} &            0 &            0 &  \frac{1}{24} &             0 &  \frac{1}{12} \\
                   0 & \frac{1}{24} &            0 &             0 &             0 &             0 \\
                   0 &            0 & \frac{1}{12} &             0 &             0 &             0 \\
        \frac{1}{24} &            0 &            0 &  \frac{1}{80} &             0 & \frac{1}{120} \\
                   0 &            0 &            0 &             0 & \frac{1}{120} &             0 \\
        \frac{1}{12} &            0 &            0 & \frac{1}{120} &             0 &  \frac{1}{20}
    \end{bmatrix}
    \quad \text{and} \quad
    A = \sqrt{2} \ts \pi
    \begin{bmatrix}
        0 &            0 &             0 &            0 &             0 &             0 \\
        0 &            0 & -\frac{1}{12} &            0 &             0 &             0 \\
        0 & \frac{1}{12} &             0 &            0 &             0 &             0 \\
        0 &            0 &             0 &            0 & -\frac{1}{60} &             0 \\
        0 &            0 &             0 & \frac{1}{60} &             0 & -\frac{1}{30} \\
        0 &            0 &             0 &            0 &  \frac{1}{30} &             0
    \end{bmatrix}
\end{equation*}
and thus obtain matrix representations of the Koopman--von Neumann generator for both the original and whitened polynomial basis sets:
\begin{equation} \label{eq:Q_matrix_HO}
    Q =
    \begin{bmatrix}
        0 &  0 &  0 &  0 &  0 & 0 \\
        0 &  0 &  2 &  0 &  0 & 0 \\
        0 & -1 &  0 &  0 &  0 & 0 \\
        0 &  0 &  0 &  0 &  2 & 0 \\
        0 &  0 &  0 & -2 &  0 & 4 \\
        0 &  0 &  0 &  0 & -1 & 0
    \end{bmatrix}
    \quad \text{and} \quad
    \widetilde{Q} =
    \begin{bmatrix}
               0 & -\sqrt{2} &    0 &    0 &    0 &    0 \\
        \sqrt{2} &         0 &    0 &    0 &    0 &    0 \\
               0 &         0 &    0 & -z_1 & -z_2 & -z_3 \\
               0 &         0 & z_1 &     0 &    0 &    0 \\
               0 &         0 & z_2 &     0 &    0 &    0 \\
               0 &         0 & z_3 &     0 &    0 &    0
    \end{bmatrix},
\end{equation}
where $ z_1, z_2, z_3 > 0 $. The matrix $ A $ is already skew-symmetric since the system is Hamiltonian, but the matrix $ Q $ is not as we still have to orthonormalize the basis functions to obtain the skew-symmetric matrix $ \widetilde{Q} $. As the basis set defines an invariant subspace for $\mathcal{Q}$, the matrix $\widetilde{Q}$ can be used to exactly propagate wavefunctions contained in the linear span of the basis functions, and to compute select eigenvalues of $ \mathcal{Q} $.

Indeed, the spectrum of $ \widetilde{Q} $ is $ \sigma(\widetilde{Q}) = \big \{ 0, \pm \mathrm{i} \ts \sqrt{2}, \pm \mathrm{i} \ts 2 \sqrt{2} \big\} $, where the eigenvalue $ \nu = 0 $ has multiplicity two, corresponding to the first six combinations $ \mathrm{i} \sqrt{2}(k_1 - k_2) $ of the principal eigenvalues. The resulting propagator $e^{t\widetilde{Q}}$ is now unitary. More precisely, given the coefficients $ c $ defining the function $ \psi(x) = c^\top \widetilde{\phi}(x) $ at time $ 0 $, we obtain the wavefunction $ \psi(x) = (e^{t \ts \widetilde{Q}} \ts c)^\top \widetilde{\phi}(x) $ at time~$ t $.

\begin{figure}
    \centering
    \begin{minipage}[t]{0.36\linewidth}
        \centering
        \subfiguretitle{(a)}
        \vspace*{1ex}
        \includegraphics[width=0.99\linewidth]{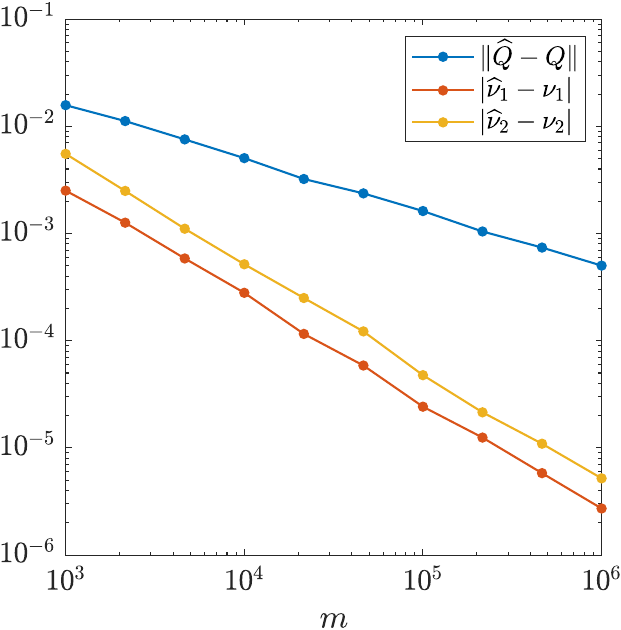}
    \end{minipage}
    \hspace*{6ex}
    \begin{minipage}[t]{0.3\linewidth}
        \centering
        \subfiguretitle{(b)}
        \vspace*{1ex}
        \includegraphics[width=0.99\linewidth]{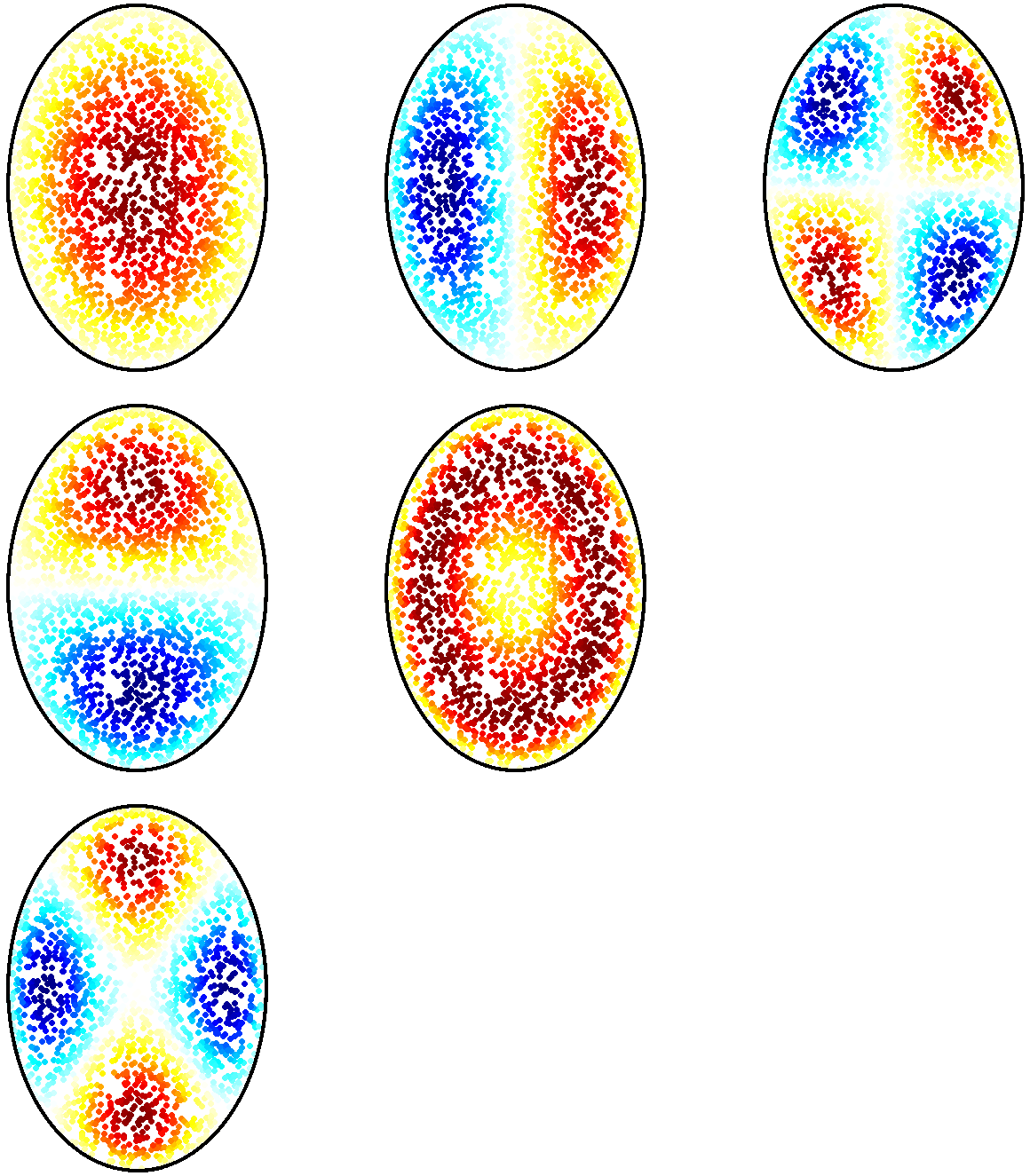}
    \end{minipage}
    \hspace*{1ex}
    \begin{minipage}[t]{0.079\linewidth}
        \centering
        \vspace*{0.5ex}
        \includegraphics[width=0.99\linewidth]{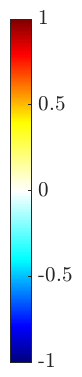}
    \end{minipage}
    \caption{(a) The convergence of the estimated matrix $ \widehat{Q} $ to the true matrix representation $ Q $ is approximately $ \mathcal{O}(m^{\nicefrac{-1}{2}}) $, which is the expected Monte Carlo convergence rate. The observed convergence to the analytically computed eigenvalues $ \nu_1 = \mathrm{i}\sqrt{2} $ and $ \nu_2 = \mathrm{i} \ts 2 \sqrt{2} $ seems to be roughly $ \mathcal{O}(m^{-1}) $, but there is currently no theoretical justification for the faster convergence. (b)~Numerically computed eigenfunctions for $ m = 2000 $, cf.\ Figure~\ref{fig:UO eigenfunctions}\ts(b). We obtain only approximations of the eigenfunctions $ f_\lambda(x) = f_0(x) \ts f_{\lambda_1}(x)^{k_1} \ts f_{\lambda_2}(x)^{k_2} $ with $ k_1 + k_2 \le 2 $ due to the chosen dictionary, corresponding to the upper left triangle in Figure~\ref{fig:UO eigenfunctions}\ts(b).}
    \label{fig:UO convergence}
\end{figure}

We also verify the accuracy of numerical estimates of the Koopman--von Neumann generator based on the gEDMD approach outlined in Section~\ref{sec:approximation}. We first uniformly sample $ m $ data points $ x^{(l)} $ in $ \Omega $, define $ \dot{x}^{(l)} = b\big(x^{(l)}\big) $, and then compute $ \widehat{Q} $ and its eigenvalues and eigenvectors, which then determine the eigenfunctions.\!\footnote{The time derivatives for the training data points could also be estimated using finite difference approximations if we only have access to trajectory data. However, since we are interested in the convergence of the method itself, we want to avoid additional numerical errors.} The numerical errors for different values of $ m $ as well as the eigenfunctions obtained for $ m = 2000 $ are shown in Figure~\ref{fig:UO convergence}.

\paragraph{Quantum circuit representation.}

We also use this example to show how the projected Koopman--von Neumann operator gives rise to a quantum circuit representation, which can be realized on a quantum computer. The matrix $ \widetilde{Q} $ in~\eqref{eq:Q_matrix_HO} and hence also $ e^{t \ts \widetilde{Q}} $ consist of a block of size two and a block of size four. The $ 2 \times 2 $ block of the propagator is
\begin{equation*}
    \begin{bmatrix}
        \cos(\sqrt{2} \ts t) & -\sin(\sqrt{2} \ts t) \\
        \sin(\sqrt{2} \ts t) & \cos(\sqrt{2} \ts t)
    \end{bmatrix}
\end{equation*}
and can be described by the single-quantum gate $ R_y(2 \sqrt{2} \ts t) $. Additionally, defining $ r = \sqrt{z_1^2 + z_2^2 + z_3^2} $, $ s = \sqrt{z_2^2+z_3^2} $, and the angles
\begin{equation*}
    \theta = 2\arctan\left(\frac{s}{z_1}\right),\quad
    \phi = 2\arctan\left(\frac{z_2}{z_3}\right),\quad
    \beta = 2 \ts r \ts t,
\end{equation*}
we can express the $ 4 \times 4 $ block of the propagator as a quantum circuit consisting only of Pauli-$ X $/NOT gates and controlled $ R_y $ rotations. The derivation of the circuit representation can be found in Appendix~\ref{app:circuit representation}. The overall quantum circuit is shown in Figure~\ref{fig:circuit representation}. Note that all preprocessing steps including the computation of the Galerkin matrices $ G $ and $ A $, the whitening transformation, and the calculation of the parameters are carried out classically. The quantum part is then restricted to the unitary time evolution. More precisely, the $ 2 \times 2 $ block is represented by a single $ R_y $ gate acting on one qubit, while the $ 4 \times 4 $ block requires two qubits, five controlled $ R_y $ rotations, and two Pauli-$ X $ gates. For the discretization considered here, the circuit size is constant, with the propagation time entering only through the rotation angles. However, this low resource count is a result of the specific block structure of the considered example and does not imply a comparable scaling behavior for other discretizations. While there are algorithms to decompose arbitrary unitary matrices into quantum gates, the number of gates in general scales exponentially with the number of qubits~\cite{PB11}. We do not claim any quantum advantage in the present setting. The purpose of this construction is rather to illustrate that a simple classical dynamical system can be encoded within a quantum-circuit framework.

\begin{figure}
    \centering
    \begin{quantikz}[column sep=0.35cm]
        \lstick{$q_1$} & \qw \gategroup[wires=1,steps=7,style={dashed,rounded corners,inner sep=2pt}, label style={label position=below,anchor=north,yshift=-0.25cm}]{} & \qw & \qw & \gate{R_y(2 \sqrt{2} \ts t)} & \qw & \qw & \qw & \qw \\[2ex]
        \lstick{$q'_1$} &
        \ctrl{1}
        \gategroup[wires=2,steps=2,style={dashed,rounded corners,inner xsep=0pt}, label style={label position=below,anchor=north,yshift=-0.25cm}]{$V$}
        & \gate{R_y(-\theta)}  &
        \gate{X} \gategroup[wires=2,steps=3,style={dashed,rounded corners,inner xsep=0pt}, label style={label position=below,anchor=north,yshift=-0.25cm}]{$U$} & \ctrl{1} & \gate{X} &
        \gate{R_y(\theta)} \gategroup[wires=2,steps=2,style={dashed,rounded corners,inner xsep=0pt}, label style={label position=below,anchor=north,yshift=-0.25cm}]{$V^\top$}
        & \ctrl{1}  &
        \qw \\
        \lstick{$q'_2$} &
         \gate{R_y(\phi)} & \ctrl{-1} &
        \qw & \gate{R_y(\beta)} & \qw &
         \ctrl{-1} & \gate{R_y(-\phi)} &
        \qw
        \end{quantikz}
    \caption{Quantum circuit representation of the propagator $ e^{t \widetilde{Q}} $ consisting of two separate circuits. The $ 2 \times 2 $ block is represented by one $ R_y $ gate acting on a single qubit, while the $ 4 \times 4 $ block can be decomposed into controlled $ R_y $ gates and Pauli-$ X $ gates acting on two qubits.}
    \label{fig:circuit representation}
\end{figure}
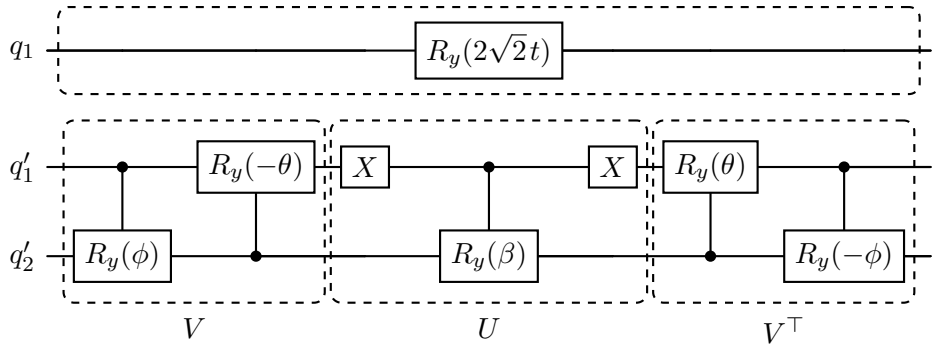

\paragraph{Random feature basis.}

In addition to the polynomial dictionary, we also evaluate numerical approximations of the Koopman--von Neumann generator by a more system-agnostic basis set. We introduce a basis of $ n $ tapered random Fourier features (RFFs)
\begin{equation} \label{eq:tapered_rff}
    \phi_i(x) = f_0(x) \cos(\omega_i^\top x + \theta_i).
\end{equation}
The frequencies $ \omega_i $ are drawn from a normal distribution with bandwidth $ \sigma^{-1} $, which is the spectral measure of a Gaussian kernel with bandwidth $ \sigma $. Moreover, $ \theta_i \in [0, 2\pi] $ are uniform random phase shifts, and $ f_0 $ is the conservation law defined in~\eqref{eq:conservation_law_undamped}, again ensuring that the basis functions vanish exactly on the boundary. The random features approximate the Gaussian kernel in the limit of infinite $ n $.

In our experiments, we set the Gaussian bandwidth to $ \sigma = 0.5 $, and use $ n = 300 $ random features. In Figure~\ref{fig:UO RFFs}\ts(a), we verify that gEDMD models trained using this basis can accurately predict the (tapered) system state variables $ f_0(x) \ts x_1 $ and $ f_0(x) \ts x_2 $. To this end, we learn gEDMD models using increasing amounts of uniform training data in $ \Omega $ and predict the system state over one period of the system. As expected, the prediction error decreases with increasing training data size at a roughly inverse square root rate. We also track the maximal change of $L^2$-norm of the predicted wavefunction compared to the initial condition, thereby verifying that the learned time evolution is unitary. The deviation remains on the level of floating point errors for all experiments. Finally, we confirm that the spectrum of the Koopman--von Neumann generator can be correctly identified by the RFF models. We see in Figure~\ref{fig:UO RFFs}\ts(b) that the eigenvalues $ \nu \in \{\pm i\sqrt{2}k\} $ for $ 0 \leq k \leq 5 $ are accurately recovered. We also notice that the method is now subject to spectral pollution, but spurious eigenvalues can be reliably removed by applying the \emph{residual score} introduced in~\cite{Colbrook23, CAS23}.

\begin{figure}
    \centering
    \begin{minipage}[t]{0.45\linewidth}
        \centering
        \subfiguretitle{(a)}
        \includegraphics[width=0.99\linewidth]{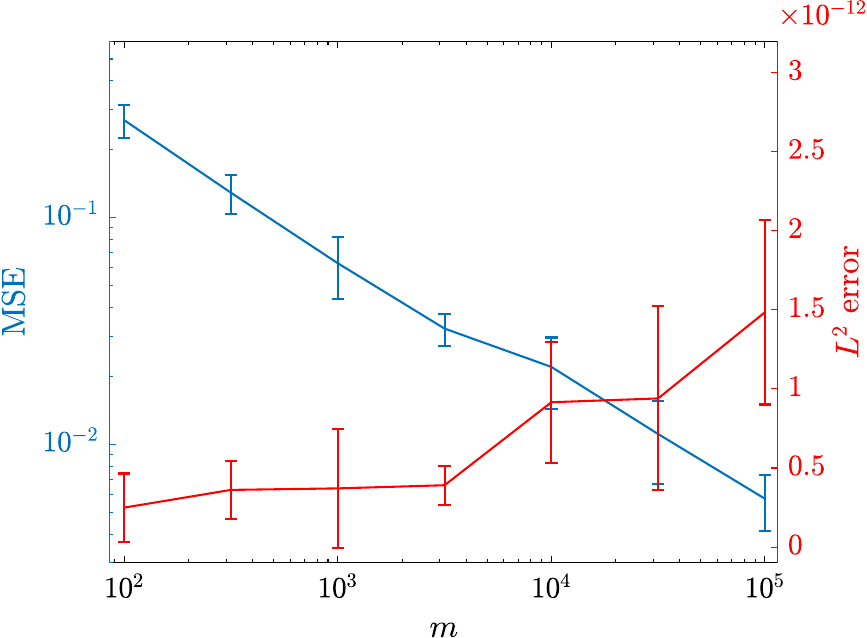}
    \end{minipage}
    \hspace*{6ex}
    \begin{minipage}[t]{0.45\linewidth}
        \centering
        \subfiguretitle{(b)}
        \vspace*{1.4ex}
        \includegraphics[width=0.904\linewidth]{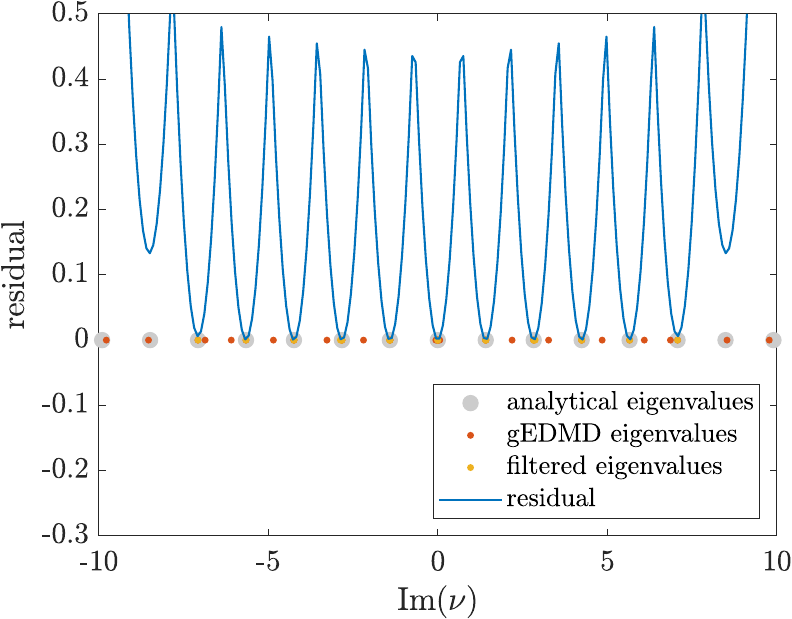}
    \end{minipage}
    \caption{(a) Mean-squared error for the state variables $ x_1 $ and $ x_2 $ predicted by gEDMD models with RFF basis~\eqref{eq:tapered_rff} over one period of the undamped oscillator and $ L^2 $ error as a function of the number of training data points $ m $. The error bars correspond to ten independent trials. (b)~Comparison of the imaginary parts of the analytically computed eigenvalues (gray dots), the gEDMD eigenvalues for $ m = 10^5 $ (red dots), and the filtered eigenvalues (yellow dots) for which the residual score (blue curve) from~\cite{Colbrook23, CAS23} is less than $ 10^{-2} $.}
    \label{fig:UO RFFs}
\end{figure}

\subsection{Damped oscillator}

Let us now consider the damped oscillator. We use this system as a first example of non-Hamiltonian dynamics, where the Koopman and Koopman--von Neumann generators describe genuinely different quantities. To illustrate the evolution of wavefunctions under the Koopman--von Neumann equation, we use a high-accuracy PDE integrator, thus minimizing numerical and statistical errors for this example. Naturally, this procedure is not scalable to higher-dimensional systems.

We decompose the domain chosen in Section~\ref{sec:benchmark problems} with the aid of \href{https://gmsh.info}{Gmsh} \cite{Gmsh} into $ 447 $ triangles and then compute the matrices $ A $ and $ G $ using \href{https://fenicsproject.org}{FEniCS} \cite{DOLFINx}. The dictionary in this case comprises piecewise linear functions. We define the initial wavefunction to be a superposition of two Gauss-like functions and then simulate its evolution in time. The numerical results are shown in Figure~\ref{fig:DO simulation}. Due to the damping, the Gauss peaks, as expected, quickly spiral towards the origin. At the same time, the height of the peaks increases and the bandwidth decreases. The simulation becomes unstable for larger times $ t $ as the wavefunction cannot be accurately resolved using the chosen basis anymore. For smaller values of the damping coefficient $ \gamma $, the spiraling effect would be more pronounced. The results show that the dynamics of non-Hamiltonian systems is described accurately by wavefunctions evolving under the Koopman--von Neumann equation, whose propagator can, using again the linear transformation of the basis functions, be approximated by unitary matrices and hence quantum circuits.

\begin{figure}
    \centering
    \begin{minipage}[t]{0.18\linewidth}
        \centering
        \subfiguretitle{(a)}
        \vspace*{1ex}
        \includegraphics[width=0.99\linewidth]{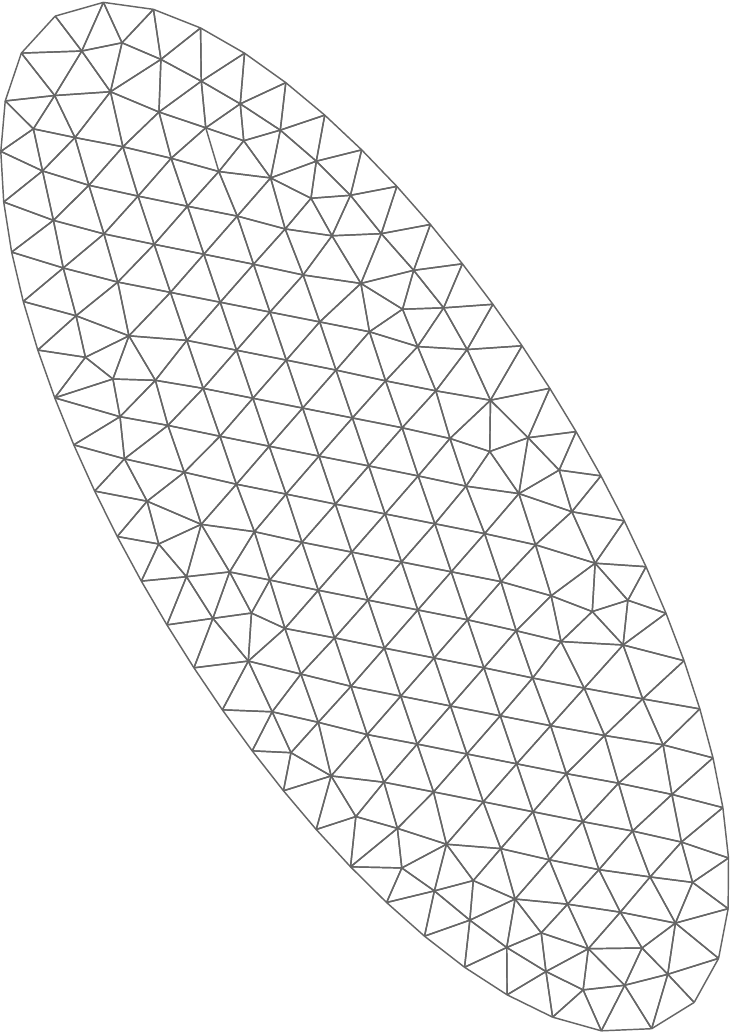}
    \end{minipage}
    \begin{minipage}[t]{0.18\linewidth}
        \centering
        \subfiguretitle{(b)}
        \vspace*{1ex}
        \includegraphics[width=\linewidth]{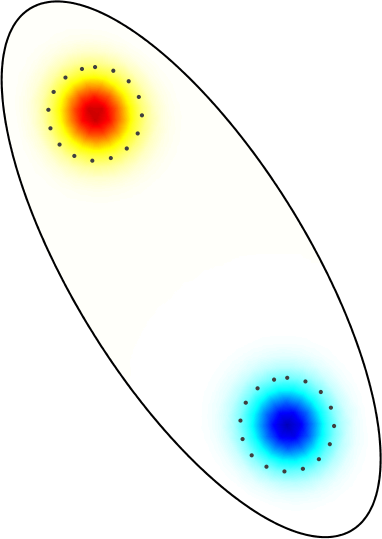}
    \end{minipage}
    \begin{minipage}[t]{0.18\linewidth}
        \centering
        \subfiguretitle{(c)}
        \vspace*{1ex}
        \includegraphics[width=\linewidth]{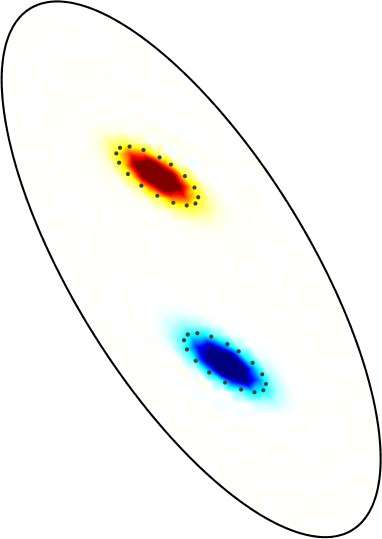}
    \end{minipage}
    \begin{minipage}[t]{0.18\linewidth}
        \centering
        \subfiguretitle{(d)}
        \vspace*{1ex}
        \includegraphics[width=\linewidth]{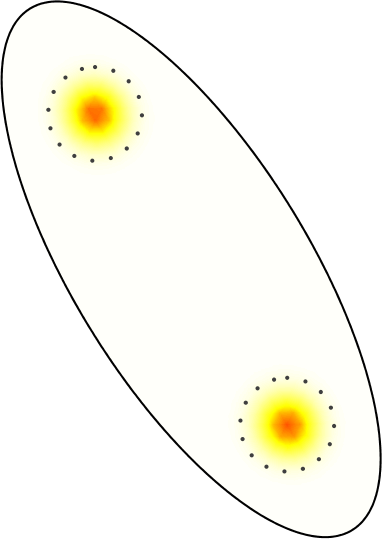}
    \end{minipage}
    \begin{minipage}[t]{0.18\linewidth}
        \centering
        \subfiguretitle{(e)}
        \vspace*{1ex}
        \includegraphics[width=\linewidth]{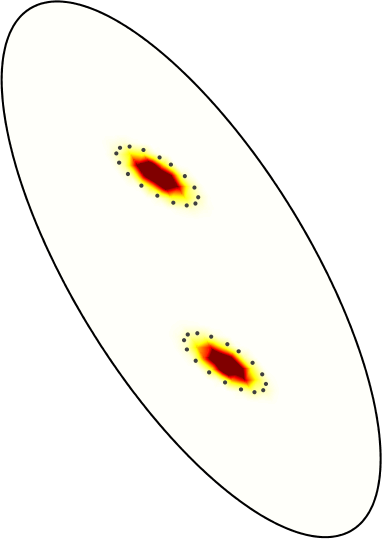}
    \end{minipage}
    \caption{(a) Gmsh triangulation of the domain $ \Omega = \big\{ x \in \R^2: x_1^2 + x_1 \ts x_2 + \frac{1}{2} \ts x_2^2 < 1 \big\} $ chosen for the damped oscillator. (b) Initial wavefunction $ \psi $ at time $ t = 0 $. (c) Wavefunction $ \psi $ at $ t = \frac{1}{2} $. The black dots, shown for the sake of comparison, represent individual particles. (d) Corresponding probability density $ \rho $ at $ t = 0 $. (e) Density $ \rho $ at $ t = \frac{1}{2} $.}
    \label{fig:DO simulation}
\end{figure}

\subsection{Lotka--Volterra model}

As a final example, we consider the nonlinear Lotka--Volterra model introduced in Section~\ref{sec:benchmark problems}. We track the evolution of a Gaussian-like wavefunction centered at $ \mu_0 = \big[\frac{1}{2}, \frac{1}{2}\big]^\top $ with isotropic bandwidth $ \sigma_0 = \sqrt{0.02} $.

\paragraph{FEM basis.}

As in the previous example, we compute a high-accuracy solution of the Koopman--von Neumann equation using finite elements. We first approximate the boundary $ \partial \Omega $ by B-splines and then triangulate the domain. Gmsh generates 3524 nodes and 6959 elements. The mass and stiffness matrices are again computed using FEniCS. We then simulate the Koopman--von Neumann equation from $ t = 0 $ to $ t = 7 $. The numerical results, shown in Figure~\ref{fig:LV simulation}\ts(a)--(d), illustrate that the numerically computed Koopman--von Neumann wavefunction faithfully describes the nonlinear oscillation. From the wavefunction, we can recover the corresponding probability density using Born's rule.

\begin{figure}
    \centering
    \begin{minipage}[t]{0.24\linewidth}
        \centering
        \subfiguretitle{(a) $ t = 0 $}
        \vspace*{1ex}
        \includegraphics[width=\linewidth]{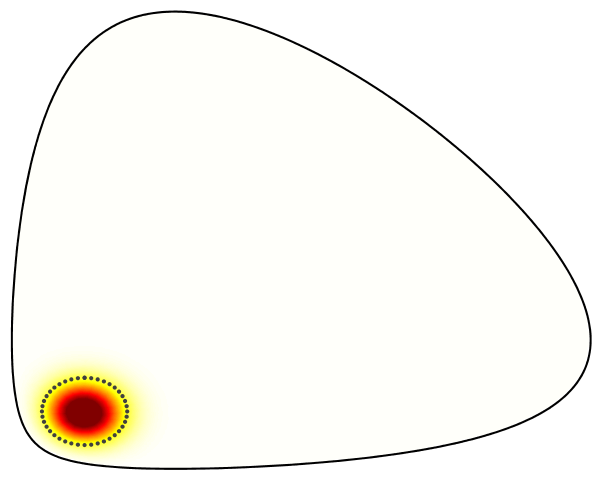}
    \end{minipage}
    \begin{minipage}[t]{0.24\linewidth}
        \centering
        \subfiguretitle{(b) $ t = 2 $}
        \vspace*{1ex}
        \includegraphics[width=\linewidth]{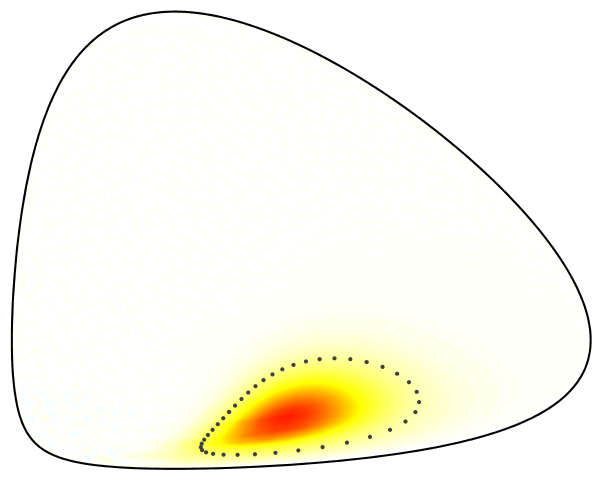}
    \end{minipage}
    \begin{minipage}[t]{0.24\linewidth}
        \centering
        \subfiguretitle{(c) $ t = 4 $}
        \vspace*{1ex}
        \includegraphics[width=\linewidth]{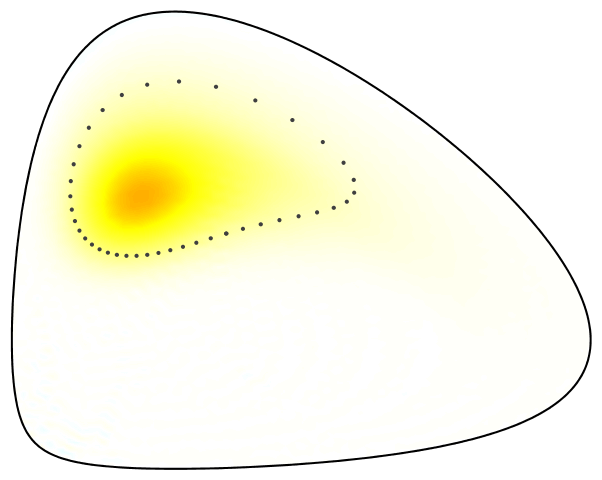}
    \end{minipage}
    \begin{minipage}[t]{0.24\linewidth}
        \centering
        \subfiguretitle{(d) $ t = 6 $}
        \vspace*{1ex}
        \includegraphics[width=\linewidth]{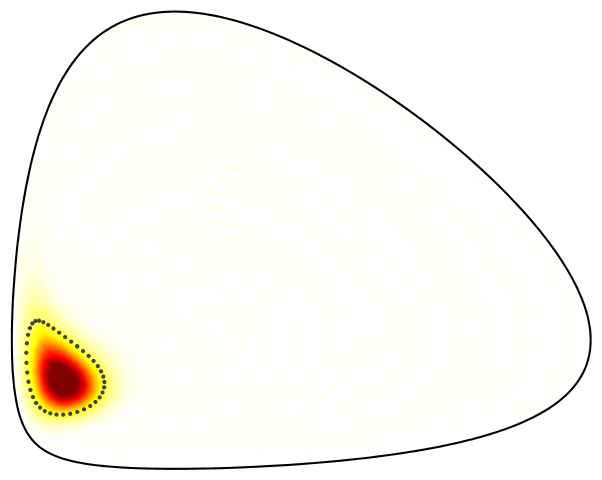}
    \end{minipage} \\[2ex]
    \begin{minipage}[t]{0.24\linewidth}
        \centering
        \subfiguretitle{(e) $ t = 0 $}
        \vspace*{1ex}
        \includegraphics[width=\linewidth]{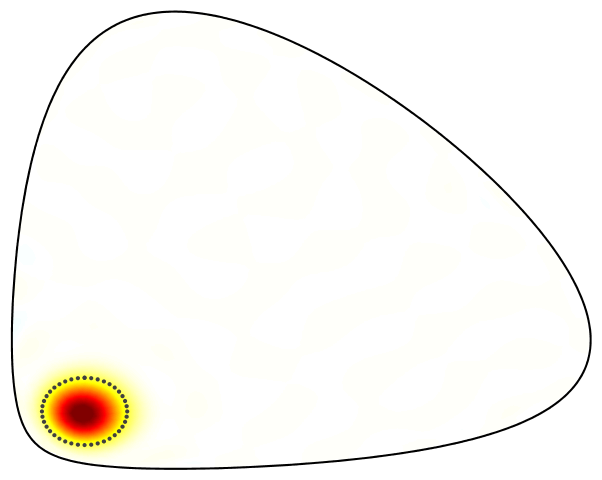}
    \end{minipage}
    \begin{minipage}[t]{0.24\linewidth}
        \centering
        \subfiguretitle{(f) $ t = 2 $}
        \vspace*{1ex}
        \includegraphics[width=\linewidth]{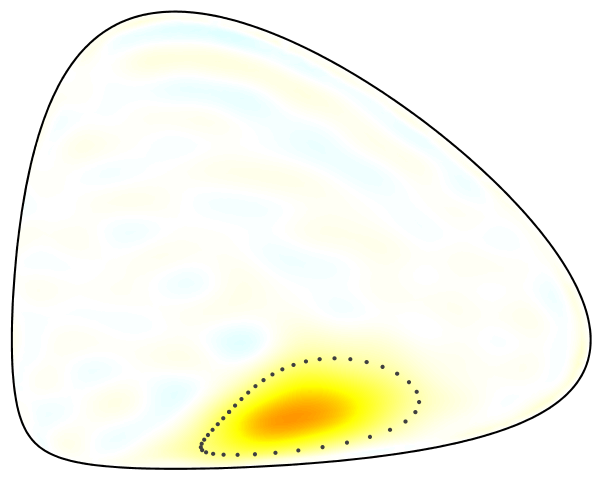}
    \end{minipage}
    \begin{minipage}[t]{0.24\linewidth}
        \centering
        \subfiguretitle{(g) $ t = 4 $}
        \vspace*{1ex}
        \includegraphics[width=\linewidth]{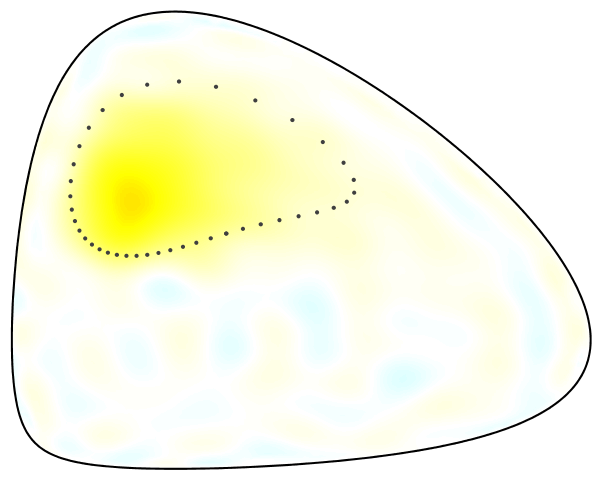}
    \end{minipage}
    \begin{minipage}[t]{0.24\linewidth}
        \centering
        \subfiguretitle{(h) $ t = 6 $}
        \vspace*{1ex}
        \includegraphics[width=\linewidth]{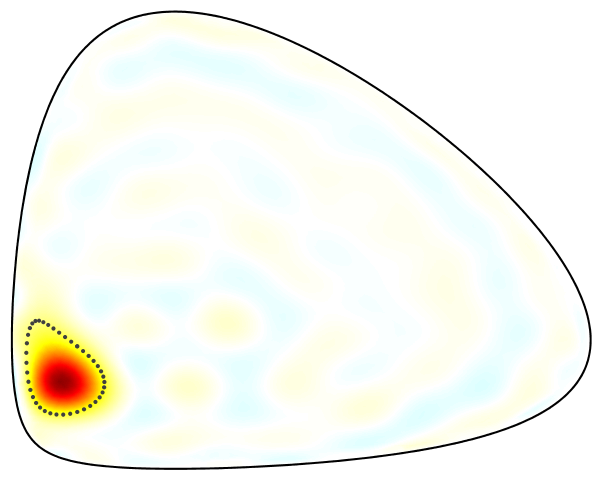}
    \end{minipage}
    \caption{Solution of the Koopman--von Neumann equation associated with the Lotka--Volterra model at different times $ t $ using an FEM basis (top row) and an RFF basis (bottom row). The initial wavefunction $ \psi $ is distorted by the dynamics and rotates in a counterclockwise direction. Note that the shape of the wavefunction will be different after the completion of one cycle since the velocities of the particles depend on the distance from the fixed point $ x = [1, 1]^\top $. The black dots again represent individual particles. The RFF basis leads to small numerical artifacts---even at time $ t = 0 $ since the initial condition needs to be approximated by tapered random Fourier features---but remains stable.}
    \label{fig:LV simulation}
\end{figure}

\paragraph{Random feature basis.}

As for the undamped oscillator, we also compute a data-driven solution to the Koopman--von Neumann equation using a random feature basis set. We again use cosine waves with random frequencies corresponding to a Gaussian kernel with bandwidth $ \sigma $ and uniform phase shifts. This time, each basis function is multiplied by an exponential tapering function
\begin{equation*}
    \eta(x) = \exp\left(-\frac{1}{k f_0(x)^2}\right),
\end{equation*}
to enforce the boundary conditions, where $ f_0 $ is the conservation law~\eqref{eq:conservation_lv} for the Lotka--Volterra domain. The exponential tapering function decays smoothly to zero as $x$ approaches the boundary. In the numerical experiments, we set $ \sigma = 0.2 $, $ n = 1000 $, $ k = 5000 $, and use $m = 50\ts000$ uniformly sampled training points. Panels\ts(e)--(h) in Figure~\ref{fig:LV simulation} show the wavefunction predicted by the RFF model. We see that the nonlinear rotation around the domain is correctly captured by the RFF wavefunction. The RFF-based solution introduces small scale fluctuations that the FEM-based model does not display. These artifacts could likely be removed by a more customized basis set and more training data. Nevertheless, this shows that even for our quite agnostic basis set, the mean-squared error between both models at training data size $ m = 100\ts000 $ is only on the order of $ 10^{-4} $, see the referenced code repository for details.

\section{Conclusion}
\label{sec:conclusion}

The Koopman--von Neumann equation describes the evolution of wavefunctions associated with classical dynamical systems. We analyzed the relationships between conventional transfer operators, such as the Koopman generator and Perron--Frobenius generator, and the Koopman--von Neumann framework. The main advantage is that the Koopman--von Neumann generator is skew-adjoint even for non-Hamiltonian systems and the corresponding propagator is unitary. This allows us to represent Galerkin approximations of the operator by unitary matrices, which in turn can be represented by quantum circuits. The results show that the definition of the domain and the choice of basis functions are crucial. Provided that the basis functions satisfy the homogeneous Dirichlet boundary conditions, we can compute approximations of the Koopman--von Neumann generator using either finite element methods or a generalization of gEDMD. For linear systems with conservation laws, it is possible to explicitly construct invariant subspaces and to compute eigenfunctions. The proposed method also allows us to propagate wavefunctions associated with nonlinear and non-Hamiltonian systems. However, if the probability accumulates---for instance in stable fixed points---and converges to a Dirac distribution, this cannot be accurately represented by a set of finitely many basis functions, which then causes numerical instabilities.

So far, we only applied the proposed methods to low-dimensional benchmark problems. Both the finite element method and the gEDMD-based approach suffer from the curse of dimensionality. That is, for high-dimensional systems, approximating the Koopman--von Neumann operator and its eigenvalues and eigenfunctions with sufficient accuracy would require a prohibitively large dictionary. There are several commonly used techniques to mitigate this issue, which can be roughly divided into two different classes: (1) Instead of working with an explicitly defined dictionary, kernel-based methods \cite{WRK15, KSM19, DG20} generate potentially infinite-dimensional feature spaces and can be applied to more complex problems. Kernels can also be efficiently approximated using \emph{random Fourier features} as shown in~\cite{NK23} (and the examples above). Alternatively, the dictionary can be parametrized by a deep neural network, which is then trained based on appropriate loss functions that are typically related to either the reconstruction error or implied timescales~\cite{LDBK17, MPWN18}. Randomized neural networks, where the weights of the hidden layers are randomly selected and only the output layer is optimized, can be used in the same way and are less computationally demanding~\cite{TLK25}. (2) The second frequently used approach, which has in particular been successfully applied to high-dimensional protein-folding problems, is to first project the data onto a dynamically relevant lower-dimensional subspace using, for example, \emph{time-lagged independent component analysis} and to then learn the operators and their eigenvalues and eigenfunctions in the reduced collective-variable space~\cite{PPGDN13, NKPMN14, SKH23}. This implicitly assumes that there are only a few dominant eigenvalues representing the slowest timescales, followed by a spectral gap, so that the projected dynamics retain the metastable behavior. Future work includes extending the aforementioned techniques to the Koopman--von Neumann operator and applying the resulting algorithms to higher-dimensional systems. Additionally, it will be necessary to extend the range of admissible sampling regimes beyond uniform sampling, and to derive convergence results and error bounds akin to~\cite{KM18, NPPSW23, ZZ23, LLLK24}.

Another open question is which steps of the overall process---defining and evaluating basis functions, constructing matrix representations of the Koopman--von Neumann generator and operator, and computing eigenvalues and eigenvectors or evolving wavefunctions---should be carried out on a classical computer and which on a quantum computer. Closely related to this is the question of how to define basis functions and discretization schemes that are suitable for quantum hardware. It remains to be investigated which discretization methods are most appropriate for a direct quantum implementation of the Koopman--von Neumann framework, how the resulting matrix representations can be decomposed efficiently into quantum circuits, and which classes of dynamical systems admit particularly suitable quantum representations. Our future research will also include the analysis of the scaling behavior of such approaches in order to identify problem classes and system sizes for which a quantum advantage over classical methods might be possible.

\section*{Data availability}

The code and examples that support the findings presented in this paper are publicly available at
\url{https://github.com/sklus/d3s} and \url{https://github.com/fnueske/kvn_ode}, see also \cite{KNG26_data}.

\section*{Acknowledgments}

We thank Paul Bergold for many interesting discussions about Koopman--von Neumann mechanics and Sofie Verhees and Ashwin Nayak for their FEniCS support. S.K.\ was funded by a Leverhulme Trust Research Fellowship.

\bibliographystyle{unsrturl}
\bibliography{KvN}

\appendix

\section{Comparison with the QECD approach}
\label{app:QECD}

The \emph{quantum embedding of classical dynamics} (QECD) approach described in~\cite{GOPSS22, FGS24, DM26} allows for an embedding of classical states and observables into a quantum system that faithfully represents the statistics of the classical system. Here, we show how the approach outlined in Section~\ref{subsec:kvn_mechanics} can be understood as an extension of QECD to systems with a non-unitary Koopman operator.

Recall from Section~\ref{subsec:koopman_pf_ops} that observables are measurable functions $ f \colon \Omega \mapsto \R $ in a suitable function space, for example in $ L^\infty(\Omega) $. Observables evolve under the Koopman semigroup $ f_t = \mathcal{K}^t f = f \circ \Phi^t $. Likewise, states of the system are probability densities $ p \in L^1(\Omega) $, evolved in time by the Perron--Frobenius operator, i.e., $ p_t = \mathcal{P}^t p $. A quantum embedding is a representation of the classical system with the following elements: (1) observables $f$ are identified with quantum observables (self-adjoint operators) $\mathcal{A}_f$ acting on a Hilbert space $H$; (2) states $p$ are identified with density operators $\rho$ acting on $H$; (3) observables and quantum observables evolve in time in such a way that expected measurements reproduce those of the classical system.

QECD achieves these goals for an ergodic measure-preserving system on the state space $ \Omega $ by (i)~choosing $ H = L^2(\Omega) $; (ii) representing observables as multiplication operators on $ H $ via $ \mathcal{A}_f \varphi = f \varphi $; (iii) representing probability densities with density operators $ \rho_p = \innerprod{p^{\nicefrac{1}{2}}}{\cdot} p^{\nicefrac{1}{2}} $, which are pure states with wavefunction $ \psi = p^{\nicefrac{1}{2}} $; (iv) defining the time-evolution of the state $ \rho_0 = \rho $ by the von Neumann equation
\begin{equation*}
	\rho_t = \mathcal{P}^t \rho_0 \ts \mathcal{K}^t = e^{-\mathrm{i}t\mathcal{H}} \rho_0 \ts e^{\mathrm{i}t\mathcal{H}},
\end{equation*}
where the Hamiltonian $ \mathcal{H} = -\mathrm{i} \mathcal{L} $ is the self-adjoint version of the Koopman generator. Since $ \rho_0 $ is a pure state, this reduces to solving the standard Schrödinger equation
\begin{equation*}
	\mathrm{i} \ts \dot{\psi}_t = \mathcal{H} \psi_t,
\end{equation*}
and again identifying $ \rho_t = \innerprod{\psi_t}{\cdot\;}\psi_t $. For ergodic measure-preserving systems, the Koopman generator $ \mathcal{L} $ is skew-adjoint, hence the definition $ \mathcal{H} = -\mathrm{i} \mathcal{L} $ coincides with our definition using the Koopman--von Neumann generator $\mathcal{Q}$. We can read the framework introduced in Section~\ref{subsec:kvn_mechanics} as using the same identification of observables and states, with the only modification that the Koopman--von Neumann generator provides the correct definition of the Hamiltonian as $ \mathcal{H} = \mathrm{i} \mathcal{Q} $ to define the quantum dynamics for a non-measure-preserving system.

In both cases, the expected value of any observable can be recovered by taking the quantum mechanical expectation of a quantum observable under the time-evolved state:
\begin{equation*}
	\begin{split}
		\tr(\rho_t \mathcal{A}_f) &= \innerprod{\psi_t}{\mathcal{A}_f \ts \psi_t} = \int_\Omega \psi_t(x) \ts f(x) \ts \psi_t(x)^* \ts \diff x \\
	       &= \int_\Omega |\psi_t(x)|^2 f(x) \ts \diff x = \mathbb{E}^{p_t}[f]
	        = \mathbb{E}[f(\Phi_t)].
	\end{split}
\end{equation*}
The second-to-last equality, i.e., Born's rule for $ \psi_t $, was shown in Section~\ref{subsec:kvn_mechanics}.

\section{Proof of Lemma \ref{lem:properties of operators}}
\label{app:proof}

\begin{enumerate}[leftmargin=4ex, itemsep=0ex, topsep=0.5ex, label=\roman*)]
\item We have $ \mathcal{L}(f \ts g) = b \vdot \nabla(f \ts g) = (b \vdot \nabla f) \ts g + f \ts (b \vdot \nabla g) = (\mathcal{L}f) \ts g + f \ts (\mathcal{L}\ts g) $.
\item This follows from i) since $ \mathcal{L}(f_{\lambda_1} \ts f_{\lambda_2}) = (\lambda_1 \ts f_{\lambda_1}) \ts f_{\lambda_2} + f_{\lambda_1} \ts (\lambda_2 \ts f_{\lambda_2}) = (\lambda_1 + \lambda_2) \ts f_{\lambda_1} \ts f_{\lambda_2} $.
\item Applying the chain rule yields
\begin{equation*}
    \mathcal{L} (g \circ f_\lambda) = (b \vdot \nabla f_\lambda) \ts g' \circ f_\lambda = \lambda \ts f_\lambda \cdot g' \circ f_\lambda,
\end{equation*}
which also implies the second statement.
\item Using iii) and choosing $ g(x) = x^r $, we have
\begin{equation*}
    \mathcal{L} f_\lambda^r = \lambda f_\lambda \cdot r \ts f_\lambda^{r-1} = r \ts \lambda \ts f_\lambda^r.
\end{equation*}
\item It holds that
\begin{align*}
    \mathcal{L}^*(f_0 \ts \rho_\mu) = -(b \vdot \nabla f_0) \ts \rho_\mu - f_0 \ts (b \vdot \nabla \rho_\mu) - \div(b) \ts f_0 \ts \rho_\mu = f_0 (\mathcal{L}^* \rho_\mu) = \mu \ts (f_0 \ts \rho_\mu).
\end{align*}
\item The product rule implies that
\begin{align*}
    \mathcal{L}^* (f \ts g) &= -b \vdot \nabla (f \ts g) - \div(b) \ts f \ts g \\
        &= -(b \vdot \nabla f) \ts g - f \ts (b \vdot \nabla g) - \div(b) \ts f \ts g \\
        &= \big(\!-b \vdot \nabla f - \tfrac{1}{2} \div(b) \ts f\big) \ts g + f \ts \big(\!-b \vdot \nabla g - \tfrac{1}{2} \div(b) \ts g\big) \\
        &= (\mathcal{Q} \ts f) \ts g + f \ts (\mathcal{Q} \ts g).
\end{align*}
\item Using vi), we get $ \mathcal{L}^*(\psi_{\nu_1} \ts \psi_{\nu_2}) = (\nu_1 \ts \psi_{\nu_1}) \ts \psi_{\nu_2} + \psi_{\nu_1} \ts (\nu_2 \ts \psi_{\nu_2}) = (\nu_1 + \nu_2) \ts \psi_{\nu_1} \ts \psi_{\nu_2} $.
\item This immediately follows from the definitions of the operators since
\begin{equation*}
    \mathcal{L}^* f_\lambda = -b \vdot \nabla f_\lambda - \div(b) \ts f_\lambda = -\lambda \ts f_\lambda - \beta \ts f_\lambda = -(\lambda + \beta) \ts f_\lambda.
\end{equation*}
The result for the Koopman--von Neumann generator follows in the same way. \hfill $\square$
\end{enumerate}

\section{Derivation of the quantum gate decomposition}
\label{app:circuit representation}

We consider a matrix $ A $ of the form
\begin{equation*}
    A =
    \begin{bmatrix}
        0 & -z_1 & -z_2 & -z_3\\
        z_1 & 0 & 0 & 0\\
        z_2 & 0 & 0 & 0\\
        z_3 & 0 & 0 & 0
    \end{bmatrix}
\end{equation*}
with $ z_1, z_2, z_3 > 0 $ and define $ r = \sqrt{z_1^2 + z_2^2 + z_3^2} $. Since $ A^3 = -r^2 A $, the exponential series for $ \exp(t \ts A) $ reduces to a combination of $ I $, $ A $, and $ A^2 $, i.e.,
\begin{equation*}
\begin{split}
    \exp(tA) &= I + \frac{\sin(rt)}{r} A + \frac{1-\cos(rt)}{r^2} A^2\\
    &=
    \begin{bmatrix}
        \cos(rt) & -\frac{\sin(rt)}{r}z_1 & -\frac{\sin(rt)}{r}z_2 & -\frac{\sin(rt)}{r}z_3\\[6pt]
        \frac{\sin(rt)}{r}z_1 & 1-\frac{1-\cos(rt)}{r^2}z_1^2 & -\frac{1-\cos(rt)}{r^2}z_1 z_2 & -\frac{1-\cos(rt)}{r^2}z_1 z_3\\[6pt]
        \frac{\sin(rt)}{r}z_2 & -\frac{1-\cos(rt)}{r^2}z_1 z_2 & 1-\frac{1-\cos(rt)}{r^2}z_2^2 & -\frac{1-\cos(rt)}{r^2}z_2 z_3\\[6pt]
        \frac{\sin(rt)}{r}z_3 & -\frac{1-\cos(rt)}{r^2}z_1 z_3 & -\frac{1-\cos(rt)}{r^2}z_2 z_3 & 1-\frac{1-\cos(rt)}{r^2}z_3^2
    \end{bmatrix}.
\end{split}
\end{equation*}
Defining $ s = \sqrt{z_2^2 + z_3^2} $ and the angles
\begin{equation*}
    \theta = 2 \arctan\left(\frac{s}{z_1}\right), \quad
    \phi = 2 \arctan\left(\frac{z_2}{z_3}\right),\quad
    \beta = 2 \ts r \ts t,
\end{equation*}
we can express $\exp(tA)$ as $V^\top U \ts V$, where
\begin{equation*}
    V = {C\!R}_y^{q_2\to q_1}(-\theta)\,{C\!R}_y^{q_1\to q_2}(\phi),
\end{equation*}
and
\begin{equation*}
    U = (X^{q_1}\otimes I^{q_2})\,{C\!R}_y^{q_1\to q_2}(\beta)\,(X^{q_1}\otimes I^{q_2}).
\end{equation*}
Since all gates are real, $ V $ is orthogonal and hence $ V^{-1} = V^+ = V^\top $. Finally, we show that the circuit indeed implements $ \exp(t \ts A) $. Define $ c_\alpha = \cos(\tfrac{\alpha}{2}) $ and $ s_\alpha = \sin(\tfrac{\alpha}{2}) $ for $ \alpha \in \{\theta, \phi, \beta\} $.
The controlled rotation $ {C\!R}_y^{q_2\to q_1}(-\theta) $ has the matrix representation
\begin{equation*}
    {C\!R}_y^{q_2\to q_1}(-\theta) =
    \begin{bmatrix}
        1 & 0 & 0 & 0 \\
        0 & c_\theta & 0 & s_\theta \\
        0 & 0 & 1 & 0\\
        0 & -s_\theta & 0 & c_\theta
    \end{bmatrix},
\end{equation*}
while
\begin{equation*}
    {C\!R}_y^{q_1\to q_2}(\phi) =
    \begin{bmatrix}
        1 & 0 & 0 & 0 \\
        0 & 1 & 0 & 0 \\
        0 & 0 & c_\phi & -s_\phi \\
        0 & 0 & s_\phi & c_\phi
    \end{bmatrix}.
\end{equation*}
Hence, $ V $ is given by
\begin{equation*}
    V =
    \begin{bmatrix}
        1 & 0 & 0 & 0\\
        0 & c_\theta & s_\theta s_\phi & c_\phi s_\theta \\
        0 & 0 & c_\phi & -s_\phi \\
        0 & -s_\theta & c_\theta s_\phi & c_\theta c_\phi
    \end{bmatrix}
\end{equation*}
and $ U $ by
\begin{equation*}
    U =
    \begin{bmatrix}
        c_\beta & -s_\beta & 0 & 0 \\
        s_\beta & c_\beta & 0 & 0 \\
        0 & 0 & 1 & 0 \\
        0 & 0 & 0 & 1
    \end{bmatrix}.
\end{equation*}
A direct multiplication of the matrices gives
\begin{equation*}
    V^\top U \ts V =
    \begin{bmatrix}
        c_\beta & -c_\theta s_\beta & -s_\theta s_\phi s_\beta & -s_\theta c_\phi s_\beta\\[4pt]
        c_\theta s_\beta & c_\theta^2 c_\beta+s_\theta^2 & c_\theta s_\theta s_\phi (c_\beta-1) & c_\theta s_\theta c_\phi (c_\beta-1)\\[4pt]
        s_\theta s_\phi s_\beta & c_\theta s_\theta s_\phi (c_\beta-1) & c_\phi^2+s_\phi^2(s_\theta^2 c_\beta+c_\theta^2) & s_\phi c_\phi (s_\theta^2 c_\beta+c_\theta^2-1)\\[4pt]
        s_\theta c_\phi s_\beta & c_\theta s_\theta c_\phi (c_\beta-1) & s_\phi c_\phi (s_\theta^2 c_\beta+c_\theta^2-1) & s_\phi^2+c_\phi^2(s_\theta^2 c_\beta+c_\theta^2)
    \end{bmatrix}.
\end{equation*}
Using the above definitions together with elementary trigonometric identities, we have
\begin{equation*}
    c_\theta = \frac{z_1}{r},\quad s_\theta=\frac{\sqrt{z_2^2 + z_3^2}}{r}, \quad
    c_\phi = \frac{z_3}{\sqrt{z_2^2 + z_3^2}},\quad s_\phi = \frac{z_2}{\sqrt{z_2^2 + z_3^2}}.
\end{equation*}
Substituting these identities into the above matrix simplifies all entries to
\begin{equation*}
    V^\top U \ts V =
    \begin{bmatrix}
        c_\beta & -\frac{z_1}{r}s_\beta & -\frac{z_2}{r}s_\beta & -\frac{z_3}{r}s_\beta\\[6pt]
        \frac{z_1}{r}s_\beta & 1+\frac{z_1^2}{r^2}(c_\beta-1) & \frac{z_1 \ts z_2}{r^2}(c_\beta-1) & \frac{z_1 \ts z_3}{r^2}(c_\beta-1)\\[6pt]
        \frac{z_2}{r}s_\beta & \frac{z_1 \ts z_2}{r^2}(c_\beta-1) & 1+\frac{z_2^2}{r^2}(c_\beta-1) & \frac{z_2 \ts z_3}{r^2}(c_\beta-1)\\[6pt]
        \frac{z_3}{r}s_\beta & \frac{z_1 \ts z_3}{r^2}(c_\beta-1) & \frac{z_2 \ts z_3}{r^2}(c_\beta-1) & 1+\frac{z_3^2}{r^2}(c_\beta-1)
    \end{bmatrix},
\end{equation*}
and with $ c_\beta = \cos(r \ts t) $, $ s_\beta = \sin(r \ts  t) $ this is exactly the matrix representation of $ \exp(t \ts A) $ derived above.

\end{document}